\documentclass[10pt,twoside]{siamltex}
\usepackage{amsfonts}
\usepackage{mathrsfs}
\usepackage{amsfonts,epsfig}

\newtheorem{remark}[theorem]{Remark}

\newtheorem{example}[theorem]{Example}

\newcommand{\complex}{\mathbb{C}}

\begin{document}



\bibliographystyle{plain}
\title{Convergence on Gauss-Seidel iterative methods for linear systems with general $H-$matrices\thanks{Corresponding author: Cheng-yi Zhang, Email: chyzhang08@126.com}}

\author{
Cheng-yi Zhang\thanks{Department of Mathematics and Mechanics of
School of Science, Xi'an Polytechnic University, Xi'an, Shaanxi,
710048, P.R. China. Supported by the Science Foundation of the
Education Department of Shaanxi Province of China\ (13JK0593), the
Scientific Research Foundation\ (BS1014) and the Education Reform
Foundation\ (2012JG40) of Xi'an Polytechnic University, and National
Natural Science Foundations of China (Grant No. 11201362 and
11271297).}\and Dan Ye\thanks{School of Science, Xi'an Polytechnic
University, Xi'an, Shaanxi, 710048, P.R. China.}\and Cong-lei
Zhong\thanks {School of Mathematics and Statistics, Henan University
of Science and Technology, Luoyang, Henan, 417003, P.R. China.}\and Shuanghua Luo\thanks{Department of Mathematics and Mechanics of
School of Science, Xi'an Polytechnic University, Xi'an, Shaanxi,
710048, P.R. China. Supported by the Science Foundation of the
Education Department of Shaanxi Province of China\ (14JK1305).} }

\pagestyle{myheadings} \markboth{Cheng-yi Zhang, Dan Ye, Cong-lei
Zhong and Shuanghua Luo}{Convergence on Gauss-Seidel iterative methods for linear
systems with general $H-$matrices} \maketitle
\begin{abstract}
It is well known that as a famous type of iterative methods in numerical linear algebra, Gauss-Seidel iterative methods are convergent for linear
systems with strictly or irreducibly diagonally dominant matrices, invertible
$H-$matrices (generalized strictly diagonally dominant matrices) and Hermitian positive definite matrices. But, the same is not necessarily true for linear
systems with nonstrictly diagonally dominant matrices and general
$H-$matrices. This paper firstly proposes some necessary and sufficient conditions for convergence on
Gauss-Seidel iterative methods to establish several new theoretical results on linear systems with nonstrictly diagonally dominant
matrices and general $H-$matrices. Then, the convergence results on
preconditioned Gauss-Seidel (PGS) iterative methods for general
$H-$matrices are presented. Finally, some numerical examples are
given to demonstrate the results obtained in this paper.
\end{abstract}
\section{Introduction} \label{intro-sec}
In this paper we consider the solution methods for the system of $n$
linear equations
\begin{equation}\label{r1}
Ax=b, \end{equation} where $A=(a_{ij})\in \mathbb{C}^{n\times n}$
and is nonsingular, $b,x\in \mathbb{C}^{n}$ and $x$ unknown. Let us
recall the standard decomposition of the coefficient matrix $A\in
\mathbb{C}^{n \times n}$,
\begin{equation}\label{r1'}A=D_A-L_A-U_A,
\end{equation}
where $D_A=diag(a_{11}, a_{22},\cdots,a_{nn})$ is a diagonal matrix,
$L_A$ and $U_A$ are strictly lower and strictly upper triangular
matrices, respectively. If $a_{ii}\neq 0$ for all $i\in \langle
n\rangle=\{1,2,\cdots,n\}$, the Jacobi iteration matrix associated
with the coefficient matrix $A$ is
\begin{equation}\label{1i}
H_J=D_A^{-1}(L_A+U_A);
\end{equation}
the forward, backward and symmetric Gauss-Seidel (FGS-, BGS- and
SGS-) iteration matrices associated with the coefficient matrix $A$
are
\begin{equation}\label{1j}
H_{FGS}=(D_A-L_A)^{-1}U_A,
\end{equation}
\begin{equation}\label{1j'}
H_{BGS}=(D_A-U_A)^{-1}L_A,
\end{equation}
and
\begin{equation}\label{1jj}
\begin{array}{lll}
H_{SGS}&=&H_{BGS}H_{FGS}\\
&=&(D_A-U_A)^{-1}L_A(D_A-U_A)^{-1}L_A,
\end{array}
\end{equation}
respectively. Then, the Jacobi, FGS, BGS and SGS iterative method
can be denoted the following iterative scheme:
\begin{equation}\label{r5}
x^{(i+1)}=Hx^{(i)}+f,\ \ \ \ i=0,1,2,\cdots\cdots
\end{equation}
where $H$ denotes iteration matrices $H_J,\ H_{FGS},\ H_{BGS}$ and
$H_{SGS}$, respectively, correspondingly, $f$ is equal to
$D_A^{-1}b,\ (D_A-L_A)^{-1}b,\ (D_A-U_A)^{-1}b$ and
$(D_A-U_A)^{-1}D_A(D_A-L_A)^{-1}b$, respectively. It is well-known
that (\ref{r5}) converges for any given $x^{(0)}$ if and only if
$\rho(H)<1$ (see \cite{R.S.15}), where $\rho(H)$ denotes the
spectral radius of the iteration matrix $H$. Thus, to establish the
convergence results of iterative scheme (\ref{r5}), we mainly study
the spectral radius of the iteration matrix in the iterative scheme
(\ref{r5}).

As is well known in some classical textbooks and monographs, see
\cite{R.S.15}, Jacobli and Guass-Seidel iterative methods for linear
systems with Hermitian positive definite matrices, strictly or
irreducibly diagonally dominant matrices and invertible
$H-$matrices(generalized strictly diagonally dominant matrices) are
convergent. Recently, the class of strictly or irreducibly
diagonally dominant matrices and invertible $H-$matrices has been
extended to encompass a wider set, known as the set of {\em general
$H-$matrices}. In a recent paper, Ref.
\cite{bru2008,bru2009,bru2012}, a partition of the $n\times n$
general $H-$matrix set, $H_n$, into three mutually exclusive classes
was obtained: the Invertible class, $H_n^I$, where the comparison
matrices of all general $H-$matrices are nonsingular, the Singular
class, $H_n^S$, formed only by singular $H-$matrices, and the Mixed
class, $H_n^M$, in which singular and nonsingular $H-$matrices
coexist. Lately, Zhang in \cite{cyzhang01} proposed some necessary
and sufficient conditions for convergence on Jacobli iterative
methods for linear systems with general $H-$matrices.

A problem has to be proposed, i.e., whether Guass-Seidel iterative
methods for linear systems with nonstrictly diagonally dominant
matrices and general $H-$matrices are convergent or not. Let us
investigate the following examples.

\begin{example}\label{ex1} Assume that either $A$ or $B$ is the coefficient matrix of linear system
(\ref{r1}), where  $A=\left[
 \begin{array}{ccccc}
 \ 2 & \ 1 & 1 \\
 \ -1 & \ 2 & 1 \\
 \ -1 & \ -1 & 2
 \end{array}
 \right]$ and $B=\left[
 \begin{array}{ccccc}
 \ 2 & \ -1 & -1 \\
 \ 1 & \ 2 & -1 \\
 \ 1 & \ 1 & 2
 \end{array}
 \right]$. It is verified that both $A$ and $B$ are nonstrictly diagonally dominant
  and nonsingular. Direct computations yield that
 $\rho(H^A_{FGS})=\rho(H^B_{BGS})=1$, while
 $\rho(H^A_{BGS})=\rho(H^B_{FGS})=0.3536<1$ and
 $\rho(H^A_{SGS})=\rho(H^B_{SGS})=0.5797<1$. This shows that BGS and
 SGS iterative methods for the matrix $A$ are convergent, while the same is not FGS iterative method
 for $A$; However, FGS and
 SGS iterative methods for the matrix $B$ are convergent, while the same is not BGS iterative method
 for $B$.
\end{example}

\begin{example}\label{ex2} Assume that either $A$ or $B$ is the coefficient matrix of linear system
(\ref{r1}), where $A=\left[
 \begin{array}{ccccc}
 \ 2 & \ -1 & 1 \\
 \ 1 & \ 2 & 1 \\
 \ 1 & \ 1 & 2
 \end{array}
 \right]$ and $B=\left[
 \begin{array}{ccccc}
 \ 2 & \ -1  \\
 \ 2 & \ 1
 \end{array}
 \right]$. It is verified that $A$ is nonstrictly diagonally dominant matrix and and $B$ is a mixed $H-$matrix. Further, they are nonsingular.
 By direct computations, it is easy to get that  $\rho(H^A_{FGS})=0.4215<1,\ \rho(H^A_{BGS})=0.3536<1$
 and $\rho(H^A_{SGS})=0.3608<1$, while
 $\rho(H^B_{FGS})=\rho(H^B_{BGS})=\rho(H^B_{SGS})=1$. This shows
 that FGS, BGS and
 SGS iterative methods converge for the matrix $A$, while they
 fail to converge for the matrix $B$.
\end{example}

In fact, the matrices $A$ and $B$ in Example \ref{ex1} and Example
\ref{ex2}, respectively, are all general $H-$matrices, but are not
invertible $H-$matrices. Guass-Seidel iterative methods for these
matrices sometime may converge for some given general $H-$matrices,
but may fail to converge for other given general $H-$matrices. How
do we get the convergence on Guass-Seidel iterative methods for
linear systems with this class of matrices without direct
computations?

Aim at the problem above, some necessary and sufficient conditions
for convergence on Gauss-Seidel iterative methods are firstly
proposed to establish some new results on nonstrictly diagonally
dominant matrices and general $H-$matrices. In particular, the
convergence results on preconditioned Gauss-Seidel (PGS) iterative
methods for general $H-$matrices are presented. Futhermore, some
numerical examples are given to demonstrate the results obtained in
this paper.

The paper is organized as follows. Some notations and preliminary
results about special matrices are given in Section 2. Some special
matrices will be defined, based on which some necessary and
sufficient conditions for convergence on Gauss-Seidel iterative
methods are firstly proposed in Section 3. Some convergence results
on preconditioned Gauss-Seidel iterative methods for general
$H-$matrices are then presented in Section 4. In Section 5, some
numerical examples are given to demonstrate the results obtained in
this paper. Conclusions are given in Section 6.
\section{Preliminaries} \label{prelimi-sec}
In this section we give some notions and preliminary results about
special matrices that are used in this paper.

$\complex^{m\times n}\ (\mathbb{R}^{m\times n})$ will be used to
denote the set of all $m\times n$ complex (real) matrices.
$\mathbb{Z}$ denotes the set of all integers. Let $\alpha\subseteq
\langle n\rangle=\{1,2,\cdots,n\}\subset\mathbb{Z}$. For nonempty
index sets $\alpha,\beta \subseteq \langle n\rangle$,
$A(\alpha,\beta)$ is the submatrix of $A\in \complex^{n\times n}$
with row indices in $\alpha$ and column indices in $\beta$. The
submatrix $A(\alpha,\alpha)$ is abbreviated to $A(\alpha)$. {\em Let
$A\in \complex^{n\times n}$, $\alpha \subset \langle n\rangle$ and
$\alpha^\prime=\langle n\rangle-\alpha$. If $A(\alpha)$ is
nonsingular, the matrix
\begin{equation}\label{2.1}
A/\alpha=A(\alpha^\prime)-A(\alpha^\prime,\alpha)[A(\alpha)]^{-1}A(\alpha,\alpha^\prime)
\end{equation}
is called the Schur complement with respect to $A(\alpha)$, indices
in both $\alpha$ and $\alpha^\prime$ are arranged with increasing
order.} We shall confine ourselves to the nonsingular $A(\alpha)$ as
far as $A/\alpha$ is concerned.

Let $A=(a_{ij})\in \mathbb{C}^{m\times n}$ and $B=(b_{ij})\in
\mathbb{C}^{m\times n}$, $A\otimes B=(a_{ij}b_{ij})\in
\mathbb{C}^{m\times n}$ denotes {\em the Hadamard product} of the
matrices $A$ and $B$. A matrix $A=(a_{ij})\in \mathbb{R}^{n\times
n}$ is called {\em nonnegative} if $a_{ij}\geq 0$ for all $i,j\in
\langle n\rangle$. {\em A matrix $A=(a_{ij})\in \mathbb{R}^{n\times
n}$ is called a $Z-$matrix if $a_{ij}\leq 0$ for all $i\not= j$.} We
will use $Z_n$ to denote the set of all $n\times n$ $Z-$matrices.
{\em A matrix $A=(a_{ij})\in Z_n$ is called an $M-$matrix if $A$ can
be expressed in the form $A=sI-B$, where $B\geq 0$, and $s\geq
\rho(B)$, the spectral radius of $B$. If $s>\rho(B)$, $A$ is called
a nonsingular $M-$matrix; if $s=\rho(B)$, $A$ is called a singular
$M-$matrix.} $M_n$, $M^{\bullet}_n$ and $M^0_n$ will be used to
denote the set of all $n\times n$ $M-$matrices, the set of all
$n\times n$ nonsingular $M-$matrices and the set of all $n\times n$
singular $M-$matrices, respectively. It is easy to see that
\begin{equation}\label{2.0}
M=M^{\bullet}_n\cup M^{0}_n\ \ \ and  \ \ \ M^{\bullet}_n\cap
M^{0}_n=\emptyset.
\end{equation}
~~~~~{\em The comparison matrix of a given matrix $A=(a_{ij})\in
\complex^{n\times n}$, denoted by $\mu(A)=(\mu_{ij})$, is defined by
$$
\mu_{ij}= \left\{
\begin{array}{cc}
|a_{ii}|,\ \ & \ \ {\rm if}\ \ i=j,\\
-|a_{ij}|, \ \ & \ \ {\rm if} \ \ i\neq j.
\end{array}
\right.$$} It is clear that $\mu(A)\in Z_n$ for a matrix $A\in
\complex^{n\times n}$. The set of {\em equimodular matrices}
associated with $A$, denoted by $\omega(A)=\{B\in\complex^{n\times
n}:\ \mu(B)=\mu(A)\}$. Note that both $A$ and $\mu(A)$ are in
$\omega(A)$. {\em A matrix $A=a_{ij}\in \complex^{n\times n}$ is
called a general $H-$matrix if $\mu(A)\in M_n$ (see \cite{AB2}). If
$\mu(A)\in M^{\bullet}_n$, $A$ is called an invertible $H-$matrix;
if $\mu(A)\in M^{0}_n$ with $a_{ii}=0$ for at least one $i\in
\langle n\rangle$, $A$ is called a singular $H-$matrix; if
$\mu(A)\in M^{0}_n$ with $a_{ii}\neq0$ for all $i\in \langle
n\rangle$, $A$ is called a mixed $H-$matrix.} $H_n$, $H^{I}_n$,
$H^{S}_n$ and $H_n^M$ will denote the set of all $n\times n$ general
$H-$matrices, the set of all $n\times n$ invertible $H-$matrices,
the set of all $n\times n$ singular $H-$matrices and the set of all
$n\times n$ mixed $H-$matrices, respectively (See \cite{bru2008}).
Similar to equalities (\ref{2.0}), we have
\begin{equation}\label{2.01}
H_n=H^{I}_n\cup H^{S}_n\cup H^{M}_n\ \ \ and  \ \ \ H^{I}_n\cap
H^{S}_n\cap H^{M}_n=\emptyset.
\end{equation}

{\em For $n\geq 2$, an $n\times n$ complex matrix $A$ is reducible
if there exists an $n\times n$ permutation matrix $P$ such that
\begin{equation}\label{matri22x1}
PAP^{T} = \left[
\begin{array}{ccc}
\ A_{11} & \ A_{12}\\
\ 0 & \ A_{22}
\end{array}
\right],
\end{equation}
where $A_{11}$ is an $r\times r$ submatrix and $A_{22}$ is an
$(n-r)\times (n-r)$ submatrix, where $1 \leq r<n$. If no such
permutation matrix exists, then $A$ is called irreducible. If $A$ is
a $1\times1$ complex matrix, then $A$ is irreducible if its single
entry is nonzero, and reducible otherwise.}
\begin{definition}\label{d1}
{ A matrix $A\in \complex^{n\times n}$ is called diagonally dominant
by row if
\begin{equation}\label{2}
   |a_{ii}|\geq \sum\limits_{j=1,j\neq
i}^{n}|a_{ij}|
\end{equation}
holds for all $i\in \langle n\rangle$. If inequality in (\ref{2})
holds strictly for all $i\in \langle n\rangle,$ $A$ is called
strictly diagonally dominant by row. If $A$ is irreducible and the
inequality in (\ref{2}) holds strictly for at least one $i\in
\langle n\rangle$, $A$ is called irreducibly diagonally dominant by
row. If (\ref{2}) holds with equality for all $i\in \langle
n\rangle,$ $A$ is called diagonally equipotent by row.}
\end{definition}

$D_n(SD_n,\ ID_n)$ and $DE_n$ will be used to denote the sets of all
$ n\times n$ (strictly, irreducibly) diagonally dominant matrices
and the set of all $n\times n$ diagonally equipotent matrices,
respectively.

\begin{definition}\label{d2}
{ A matrix $A\in \complex^{n\times n}$ is called generalized
diagonally dominant if there exist positive constants $\alpha_i,\ \
i\in\langle n\rangle,$ such that
 \begin{equation}\label{av}
 \alpha_i |a_{ii}| \geq \sum\limits_{j=1,j\neq i}^{n}\alpha_j|a_{ij}|
 \end{equation}
holds for all $i\in \langle n\rangle$. If inequality in (\ref{av})
holds strictly for all $i\in \langle n\rangle,$ $A$ is called
generalized strictly diagonally dominant. If (\ref{av}) holds with
equality for all $i\in \langle n\rangle,$  $A$ is called generalized
diagonally equipotent.}
\end{definition}

We will denote the sets of all $n\times n$ generalized (strictly)
diagonally dominant matrices and the set of all $n\times n$
generalized diagonally equipotent matrices by $GD_n(GSD_n)$ and
$GDE_n$, respectively.

\begin{definition}
{ A matrix $A$ is called nonstrictly diagonally dominant, if either
(\ref{2}) or (\ref{av}) holds with equality for at least one $i\in
\langle n\rangle$.}
\end{definition}

\begin{remark}\label{remark1}
{Let $A=(a_{ij})\in \complex^{n\times n}$ be nonstrictly diagonally
dominant and $\alpha=\langle n\rangle-\alpha'\subset \langle
n\rangle$. If $A(\alpha)$ is a (generalized) diagonally equipotent
principal submatrix of $A$, then the following hold:
\begin{itemize}
\item
$A(\alpha,\alpha')=0$, which shows that $A$ is reducible;
\item
$A(i_1)=(a_{i_1i_1})$ being (generalized) diagonally equipotent
implies $a_{i_1i_1}=0$.
\end{itemize}
}
\end{remark}

\begin{remark}\label{re2.5}
{  Definition 2.2 and Definition 2.3 show that
$$D_n\subset GD_n\ \ and\ \ GSD_n\subset GD_n.$$
}
\end{remark}
~~~The following will introduce the relationship of (generalized)
diagonally dominant matrices and general $H-$matrices and some
properties of general $H-$matrices that will be used in the rest of
the paper.

\begin{lemma}\label{l9} {\rm (see \cite{Z.C17,Z.X19,Z.X22,zhang26})}
Let $A\in D_n(GD_n)$. Then $A\in H^{I}_n$ if and only if $A$ has no
(generalized) diagonally equipotent principal submatrices.
Furthermore, if $A\in D_n\cap Z_n(GD_n\cap Z_n)$, then $A\in
M^{\bullet}_n$ if and only if $A$ has no (generalized) diagonally
equipotent principal submatrices.
\end{lemma}

\begin{lemma}\label{l11} {\rm (see \cite{AB2})}
 \ \ \ \ \ $SD_n\cup ID_n\subset H^{I}_n=GSD_n$
\end{lemma}

\begin{lemma}\label{lemma2.8} {\rm (see \cite{bru2008})}
 \ \ \ \ \ $GD_n\subset H_n$.
\end{lemma}

It is interested in wether $H_n\subseteq GD_n$ is true or not. The
answer is "NOT". Some counterexamples are given in \cite{bru2008} to
show that $H_n\subseteq GD_n$ is not true. But, under the condition
of "irreducibility", the following conclusion holds.

\begin{lemma}\label{0.1.7}{\rm (see \cite{bru2008})}
Let $A\in \mathbb{C}^{n\times n}$ be irreducible. Then $A\in H_n$ if
and only if $A\in GD_n$.
\end{lemma}

More importantly, under the condition of "reducibility", we have the
following conclusion.

\begin{lemma}\label{0.1.8}
Let $A\in \mathbb{C}^{n\times n}$ be reducible. Then $A\in H_n$ if
and only if in the Frobenius normal from of $A$
\begin{equation}\label{eq5a}
 \ PAP^T = \left[
 \begin{array}{cccc}
 \ R_{11} & \ R_{12} & \cdots & R_{1s} \\
 \ 0  & \ R_{22} & \cdots & R_{2s} \\
 \ \vdots & \vdots & \ddots & \vdots \\
 \ 0  & \ 0 & \cdots & R_{ss}
 \end{array}
 \right],
 \end{equation}
each irreducible diagonal square block $R_{ii}$ is generalized
diagonally dominant, where $P$ is a permutation matrix,
$R_{ii}=A(\alpha_i)$ is either $1\times 1$ zero matrices or
irreducible square matrices, $R_{ij}=A(\alpha_i,\alpha_j),\ i\neq
j,\ i,j=1,2,\cdots,s,$ further, $\alpha_i\cap\alpha_j=\emptyset\
for\ i\neq j,$ and $\cup^s_{i=1}\alpha_i=\langle n\rangle$.
\end{lemma}

The proof of this theorem follows from Lemma \ref{0.1.7} and Theorem
5 in \cite{bru2008}.

\begin{lemma}\label{le0}
A matrix $A\in H^{M}_n\cup H^{S}_n$ if and only if in the Frobenius
normal from (\ref{eq5a}) of $A$, each irreducible diagonal square
block $R_{ii}$ is generalized diagonally dominant and has at least
one generalized diagonally equipotent principal submatrix.
\end{lemma}
\begin{proof}
It follows from (\ref{2.01}), Lemma \ref{l9} and Lemma \ref{0.1.8}
that the conclusion of this lemma is obtained immediately.
\end{proof}

\section{Some special matrices and their properties}\label{definition-sec} In order to investigate convergence on Gauss-Seidel iterative
methods, some definitions of special matrices will be defined and
their properties will be proposed to be used in this paper.

\begin{definition}\label{gi0}{\rm (see \cite{cyzhang01})}
Let $E^{i\theta}=(e^{i\theta_{rs}})\in \mathbb{C}^{n\times n}$,
where $e^{i\theta_{rs}}=\cos\theta_{rs}+i\sin\theta_{rs}$,
$i=\sqrt{-1}$ and $\theta_{rs}\in \mathbb{R}$ for all $r,s\in
\langle n\rangle$. The matrix $E^{i\theta}=(e^{i\theta_{rs}})\in
\mathbb{C}^{n\times n}$ is called $\theta-$ray pattern matrix if
\begin{enumerate}
\item $\theta_{rs}+\theta_{sr}=2k\pi$ holds for all $r,s\in
\langle n\rangle,\ r\neq s$, where $k\in \mathbb{Z}$;
\item both $\theta_{rs}-\theta_{rt}=\theta_{ts}+(2k+1)\pi$ and $\theta_{sr}-\theta_{tr}=\theta_{st}+(2k+1)\pi$ hold for all $r,s,t\in
\langle n\rangle$ and $r\neq s,\ r\neq t,\ t\neq s$, where $k\in
\mathbb{Z}$;
\item $\theta_{rr}=\theta$ for all $r\in \langle n\rangle$, $\theta\in [0,2\pi)$.
\end{enumerate}
\end{definition}

\begin{definition}\label{gi0'}
Let $E^{i\psi}=(e^{i\psi_{rs}})\in \mathbb{C}^{n\times n}$, where
$e^{i\psi_{rs}}=\cos\psi_{rs}+i\sin\psi_{rs}$, $i=\sqrt{-1}$ and
$\psi_{rs}\in \mathbb{R}$ for all $r,s\in \langle n\rangle$. The
matrix $E^{i\psi}=(e^{i\psi_{rs}})\in \mathbb{C}^{n\times n}$ is
called $\psi-$ray pattern matrix if
\begin{enumerate}
\item $\psi_{rs}+\psi_{sr}=2k\pi+\psi$ holds for all $r,s\in
\langle n\rangle,\ r\neq s$, where $k\in \mathbb{Z}$;
\item $\psi_{rs}-\psi_{rt}=\psi_{ts}+(2k+1)\pi$ if $r<s<t$ or $s<t<r$ or $t<r<s$
for all $r,s,t\in \langle n\rangle$, where $k\in \mathbb{Z}$;
\item $\psi_{rs}-\psi_{rt}=\psi_{ts}-\psi+(2k+1)\pi$ if $r<t<s$ or $t<s<r$ or $s<r<t$
for all $r,s,t\in \langle n\rangle$, where $k\in \mathbb{Z}$;
\item $\psi_{rr}=0$ for all $r\in \langle n\rangle$.
\end{enumerate}
\end{definition}

\begin{definition}\label{gi0'1}
Let $E^{i\phi}=(e^{i\phi_{rs}})\in \mathbb{C}^{n\times n}$, where
$e^{i\phi_{rs}}=\cos\phi_{rs}+i\sin\phi_{rs}$, $i=\sqrt{-1}$ and
$\phi_{rs}\in \mathbb{R}$ for all $r,s\in \langle n\rangle$. The
matrix $E^{i\phi}=(e^{i\phi_{rs}})\in \mathbb{C}^{n\times n}$ is
called $\phi-$ray pattern matrix if
\begin{enumerate}
\item $\phi_{rs}+\phi_{sr}=2k\pi+\phi$ holds for all $r,s\in
\langle n\rangle,\ r\neq s$, where $k\in \mathbb{Z}$;
\item $\phi_{rs}-\phi_{rt}=\phi_{ts}-\phi+(2k+1)\pi$ if $r<s<t$ or $s<t<r$ or $t<r<s$
for all $r,s,t\in \langle n\rangle$, where $k\in \mathbb{Z}$;
\item $\phi_{rs}-\phi_{rt}=\phi_{ts}+(2k+1)\pi$ if $r<t<s$ or $t<s<r$ or $s<r<t$
for all $r,s,t\in \langle n\rangle$, where $k\in \mathbb{Z}$;
\item $\phi_{rr}=0$ for all $r\in \langle n\rangle$.
\end{enumerate}
\end{definition}

\begin{definition}\label{gi1}
Any matrix $A=(a_{rs})\in \mathbb{C}^{n\times n}$ has the following
form:
\begin{equation}\label{eq2}A=e^{i\eta}\cdot|A|\otimes
E^{i\chi}=(e^{i\eta}\cdot|a_{rs}|e^{i\chi_{rs}})\in
\mathbb{C}^{n\times n},\end{equation} where  $\eta\in \mathbb{R}$,
$|A|=(|a_{rs}|)\in \mathbb{R}^{n\times n}$ and
$E^{i\chi}=(e^{i\chi_{rs}})\in \mathbb{C}^{n\times n}$,
$\chi_{rs}\in \mathbb{R}$ for $r,s\in \langle n\rangle.$ The matrix
$E^{i\chi}$ is called ray pattern matrix of the matrix $A$. If the
ray pattern matrix $E^{i\chi}$ of the matrix $A$ is a $\theta-$ray
pattern matrix, then $A$ is called a $\theta-$ray matrix; if the ray
pattern matrix $E^{i\chi}$ of the matrix $A$ is a $\psi-$ray pattern
matrix, then $A$ is called a $\psi-$ray matrix; and if the ray
pattern matrix $E^{i\chi}$ of the matrix $A$ is a $\phi-$ray pattern
matrix, then $A$ is called a $\phi-$ray matrix.
\end{definition}

$\mathscr{R}^{\theta}_n$, $\mathscr{U}^{\psi}_n$ and
$\mathscr{L}^{\phi}_n$ denote the set of all $n\times n$
$\theta-$ray matrices, the set of all $n\times n$ $\psi-$ray
matrices and the set of all $n\times n$ $\phi-$ray matrices,
respectively. Obviously, if a matrix $A\in \mathscr{R}^{\theta}_n$,
then $\xi\cdot A\in \mathscr{R}^{\theta}_n$ for all $\xi\in
\mathbb{C}$, the same is the matrices in $\mathscr{U}^{\psi}_n$ and
$\mathscr{L}^{\phi}_n$, respectively.

\begin{theorem}\label{thgi0} Let a matrix $A=D_A-L_A-U_A=(a_{rs})\in \mathbb{C}^{n\times n}$
with $D_A=diag(a_{11}, a_{22},\cdots,a_{nn})$. Then $A\in
\mathscr{R}^{\theta}_n$ if and only if there exists an $n\times n$
unitary diagonal matrix $D$ such that
$D^{-1}AD=e^{i\eta}\cdot[|D_A|e^{i\theta}-(|L_A|+|U_A|)]$ for
$\eta\in \mathbb{R}$.
\end{theorem}

\begin{proof} According to Definition \ref{gi1}, $A=e^{i\eta}\cdot|A|\otimes
E^{i\theta}=(e^{i\eta}\cdot|a_{rs}|e^{i\theta_{rs}}).$ Define a
diagonal matrix
$D_{\phi}=diag(e^{i\phi_1},e^{i\phi_2},\cdots,e^{i\phi_n})$ with
$\phi_r=\theta_{1r}+\phi_1+(2k+1)\pi$ for $\phi_1\in \mathbb{R}$,
$r=2,3,\cdots,n,$ and $k\in \mathbb{Z}.$ By Definition \ref{gi0},
$D^{-1}AD=e^{i\eta}\cdot[|D_A|e^{i\theta}-(|L_A|+|U_A|)]$, which
shows that the necessity is true.

The following will prove the sufficiency. Assume that there exists
an $n\times n$ unitary diagonal matrix
$D_{\phi}=diag(e^{i\phi_1},\cdots,e^{i\phi_n})$ such that
$D_{\phi}^{-1}AD_{\phi}=e^{i\eta}\cdot[|D_A|e^{i\theta}-(|L_A|+|U_A|)]$.
Then the following equalities hold:
\begin{equation}\label{eqv1}
\begin{array}{ccccc}
 \theta_{rs}&=&\phi_s-\phi_r+(2k_1+1)\pi,\\
 \theta_{sr}&=&\phi_r-\phi_s+(2k_2+1)\pi,\\
 \theta_{rt}&=&\phi_t-\phi_r+(2k_3+1)\pi,\\
 \theta_{tr}&=&\phi_r-\phi_t+(2k_4+1)\pi,
 \end{array}
\end{equation}
where $k_1,k_2,k_3,k_4\in \mathbb{Z}.$ In (\ref{eqv1}),
$\theta_{rs}+\theta_{sr}=2(k_1+k_2+1)\pi=2k\pi$ with $k=k_1+k_2+1\in
\mathbb{Z}$ and for all $r,s\in \langle n\rangle,\ r\neq s$.
Following (\ref{eqv1}), $\theta_{ts}=\phi_s-\phi_t+(2k_5+1)\pi$.
Hence, $\phi_s-\phi_t=\theta_{ts}-(2k_5+1)\pi$. Consequently,
$\theta_{rs}-\theta_{rt}=\phi_s-\phi_t+2(k_1-k_3)\pi=\theta_{ts}+[2(k_1-k_3-k_5-1)+1]\pi
\theta_{ts}+(2k+1)\pi$ for all $r,s,t\in \langle n\rangle$ and
$r\neq s,\ r\neq t,\ t\neq s$, where $k=k_1-k_3-k_5-1\in
\mathbb{Z}$. In the same method, we can prove that
$\theta_{sr}-\theta_{tr}=\theta_{st}+(2k+1)\pi$ hold for all
$r,s,t\in \langle n\rangle$ and $r\neq s,\ r\neq t,\ t\neq s$, where
$k\in \mathbb{Z}$. Furthermore, it is obvious that
$\theta_{rr}=\theta$ for all $r\in \langle n\rangle$. This completes
the sufficiency.
\end{proof}

In the same method of proof as Theorem \ref{thgi0}, the following
conclusions will be established.

\begin{theorem}\label{thgi0a1} Let a matrix $A=D_A-L_A-U_A=(a_{rs})\in \mathbb{C}^{n\times n}$
with $D_A=diag(a_{11}, a_{22},\cdots,a_{nn})$. Then $A\in
\mathscr{U}^{\psi}_n$ if and only if there exists an $n\times n$
unitary diagonal matrix $D$ such that
$D^{-1}AD=e^{i\eta}\cdot[(|D_A|-|L_A|)-e^{i\psi}|U_A|]$ for $\eta\in
\mathbb{R}$.
\end{theorem}

\begin{theorem}\label{thgi0a2} Let a matrix $A=D_A-L_A-U_A=(a_{rs})\in \mathbb{C}^{n\times n}$
with $D_A=diag(a_{11}, a_{22},\cdots,a_{nn})$. Then $A\in
\mathscr{L}^{\phi}_n$ if and only if there exists an $n\times n$
unitary diagonal matrix $D$ such that
$D^{-1}AD=e^{i\eta}\cdot[(|D_A|-|U_A|])-e^{i\phi}|L_A|$ for $\eta\in
\mathbb{R}$.
\end{theorem}

\begin{corollary}\label{thgi0'} $\mathscr{R}^{0}_n=\mathscr{U}^{0}_n=\mathscr{L}^{0}_n=\mathscr{U}^{\psi}_n\cap\mathscr{L}^{\phi}_n$.
\end{corollary}

\begin{proof} By Theorem \ref{thgi0}, Theorem \ref{thgi0a1} and
Theorem \ref{thgi0a2}, the proof is obtained immediately.
\end{proof}

\section{Convergence on Gauss-Seidel iterative methods}\label{main results-sec}
In numerical linear algebra, the Gauss-Seidel iterative method, also
known as the Liebmann method or the method of successive
displacement, is an iterative method used to solve a linear system
of equations. It is named after the German mathematicians Carl
Friedrich Gauss(1777-1855) and Philipp Ludwig von Seidel(1821-1896),
and is similar to the Jacobi method. Later, this iterative method
was developed as three iterative methods, i.e., the forward,
backward and symmetric Gauss-Seidel (FGS-, BGS- and SGS-) iterative
methods. Though these iterative methods can be applied to any matrix
with non-zero elements on the diagonals, convergence is only
guaranteed if the matrix is strictly or irreducibly diagonally
dominant matrix, Hermitian positive definite matrix and invertible
$H-$matrix. Some classic results on convergence on Gauss-Seidel
iterative methods as follows:

\begin{theorem}\label{le1}  {\rm (see \cite{kolotilina2003})}
Let $A\in SD_n\cup ID_n$. Then $\rho(H_{FGS})<1$, $\rho(H_{BGS})<1$
and $\rho(H_{SGS})<1$, where $H_{FGS},\ H_{BGS}$ and $H_{SGS}$ are
defined in (\ref{1j}), (\ref{1j'}) and (\ref{1jj}), respectively,
and therefore the sequence $\{x^{(i)}\}$ generated by FGS-, BGS- and
SGS-scheme (\ref{r5}), respectively, converges to the unique
solution of {\rm(\ref{r1})} for any choice of the initial guess
$x^{(0)}$.
\end{theorem}

\begin{theorem}\label{le2}  {\rm (see \cite{Z.X19,zhang26})}
Let $A\in H^I_n$. Then $\rho(H_{FGS})<1$, $\rho(H_{BGS})<1$ and
$\rho(H_{SGS})<1$, where $H_{FGS},\ H_{BGS}$ and $H_{SGS}$ are
defined in (\ref{1j}), (\ref{1j'}) and (\ref{1jj}), respectively,
and therefore the sequence $\{x^{(i)}\}$ generated by FGS-, BGS- and
SGS-scheme (\ref{r5}), respectively, converges to the unique
solution of {\rm(\ref{r1})} for any choice of the initial guess
$x^{(0)}$.
\end{theorem}

\begin{theorem} \label{lemma8}{\rm (see \cite{Z.X19,Z.X22})}
Let $A\in \mathbb{C}^{n\times n}$ be a Hermitian positive definite
matrix. Then $\rho(H_{FGS})<1$, $\rho(H_{BGS})<1$ and
$\rho(H_{SGS})<1$, where $H_{FGS},\ H_{BGS}$ and $H_{SGS}$ are
defined in (\ref{1j}), (\ref{1j'}) and (\ref{1jj}), respectively,
and therefore the sequence $\{x^{(i)}\}$ generated by FGS-, BGS- and
SGS-scheme (\ref{r5}), respectively, converges to the unique
solution of {\rm(\ref{r1})} for any choice of the initial guess
$x^{(0)}$.
\end{theorem}

Following, we consider convergence on Gauss-Seidel iterative methods
for general $H-$matrices. Let us investigate the case of nonstrictly
diagonally dominant matrices. By Lemma \ref{l9} and Theorem
\ref{le2}, the following conclusion is obtained.

\begin{theorem} \label{lemma8'}
Let $A\in D_n(GD_n)$. Then $\rho(H_{FGS})<1$, $\rho(H_{BGS})<1$ and
$\rho(H_{SGS})<1$, where $H_{FGS},\ H_{BGS}$ and $H_{SGS}$ are
defined in (\ref{1j}), (\ref{1j'}) and (\ref{1jj}), respectively,
i.e., the sequence $\{x^{(i)}\}$ generated by FGS-, BGS- and
SGS-scheme (\ref{r5}), respectively, converges to the unique
solution of {\rm(\ref{r1})} for any choice of the initial guess
$x^{(0)}$ if and only if $A$ has no (generalized) diagonally
equipotent principal submatrices.
\end{theorem}

Theorem \ref{lemma8'} indicates that studying convergence on
Gauss-Seidel iterative methods for nonstrictly diagonally dominant
matrices only investigates the case of (generalized) diagonally
equipotent matrices. Continuing in this direction, a lemma will be
introduced firstly to be used in this section.

\begin{lemma}\label{le2'} {\rm (see \cite{kolotilina2003,Z.X19})} Let an irreducible matrix $A\in D_n(GD_n)$.
Then $A$ is singular if and only if $D^{-1}_AA\in DE_n(GDE_n)\cap
\mathscr{R}^{0}_n$, where $D_A=diag(a_{11}, \cdots,a_{nn})$.
\end{lemma}

\begin{theorem}\label{le3} Let an irreducible matrix $A=(a_{ij})\in
GDE_2$. Then $\rho(H_{FGS})=\rho(H_{BGS})=\rho(H_{SGS})=1$, where
$H_{FGS},\ H_{BGS}$ and $H_{SGS}$ are defined in (\ref{1j}),
(\ref{1j'}) and (\ref{1jj}), respectively, and therefore the
sequence $\{x^{(i)}\}$ generated by FGS-, BGS- and SGS-scheme
(\ref{r5}), respectively, doesn't converge to the unique solution of
{\rm(\ref{r1})} for any choice of the initial guess $x^{(0)}$.
\end{theorem}

\begin{proof}
Assume $A=\left[\begin{array}{cc}
\ a_{11} &\ a_{12}\\
a_{21} &\ a_{22}
\end{array}
\right]\in GDE_2.$ By Definition \ref{d2},
$\alpha_1|a_{11}|=\alpha_2|a_{12}|$ and
$\alpha_2|a_{22}|=\alpha_1|a_{21}|$ with $a_{ij}\neq0$ and
$\alpha_i>0$ for all $i,j=1,2$. Consequently, $A\in GDE_2$ if and
only if $|a_{12}a_{21}|/|a_{11}a_{11}|=1$. Direct computations give
that
$\rho(H_{FGS})=\rho(H_{BGS})=\rho(H_{SGS})=|a_{12}a_{21}|/|a_{11}a_{22}|=1$
and consequently, the sequence $\{x^{(i)}\}$ generated by FGS-, BGS-
and SGS-scheme (\ref{r5}), respectively, doesn't converge to the
unique solution of {\rm(\ref{r1})} for any choice of the initial
guess $x^{(0)}$.
\end{proof}

\begin{lemma}\label{le4} Let $A=(a_{ij})\in DE_n\ (n\geq3)$ be irreducible.
Then $e^{i\psi}$ is an eigenvalue of $H_{FGS}$ if and only if
$D_{A}^{-1}A\in \mathscr{U}^{\psi}_n$, where
$D_{A}=diag(a_{11},a_{22},\cdots,a_{nn})$ and $\psi\in \mathbb{R}.$
\end{lemma}

\begin{proof} We prove the sufficiency firstly. Since $A=(a_{ij})\in DE_n$ is irreducible,
$a_{ii}\neq0$ for all $i\in \langle n\rangle$. Thus,
$(D_{A}+L_A)^{-1}$ exists, and consequently, $H_{FGS}$ also exists,
where $D_{A}=diag(a_{11},a_{22},\cdots,a_{nn})$. Assume
$D_{A}^{-1}A\in \mathscr{U}^{\psi}_n$. Theorem \ref{thgi0a1} shows
that there exists an unitary diagonal matrix $D$ such that
$D^{-1}(D_{A}^{-1}A)D=[(I-|D_{A}^{-1}L_A|)-e^{i\psi}|D_{A}^{-1}U_A|]$
for $\psi\in \mathbb{R}$. Hence,
$D_{A}^{-1}A=D(I-|D_{A}^{-1}L_A|)D^{-1}-e^{i\psi}D|D_{A}^{-1}U_A|D^{-1}$
and
\begin{equation}
\begin{array}{lllll}\label{c1a1}H_{FGS}&=&(D_A-L_A)^{-1}U_A=(I-D_{A}^{-1}L_A)^{-1}D_{A}^{-1}U_A\\
&=&[D(I-|D_A^{-1}L_A|)D^{-1}]^{-1}(
e^{i\psi}D|D_{A}^{-1}U_A|D^{-1})\\
&=&e^{i\psi}D[(I-|D_{A}^{-1}L_A|)^{-1}|D_{A}^{-1}U_A|]D^{-1}.
\end{array}
\end{equation}
Using (\ref{c1a1}),
\begin{equation}\label{c1b1}
\begin{array}{lllll}
\det(e^{i\psi}I-H_{FGS})&=&\det[e^{i\psi}I-e^{i\psi}D((I-|D_{A}^{-1}L_A|)^{-1}|D_{A}^{-1}U_A|)D^{-1}]\\
&=& e^{i\psi}\cdot \det
(I-(I-|D_{A}^{-1}L_A|)^{-1}|D_{A}^{-1}U_A|)\\&=&\displaystyle\frac{e^{i\psi}\cdot
\det(I-|D_{A}^{-1}L_A|-|D_{A}^{-1}U_A|)}{\det(I-|D_{A}^{-1}L_A|)}\\
&=&\displaystyle\frac{e^{i\psi}\cdot
\det\mu(D_{A}^{-1}A)}{\det(I-|D_{A}^{-1}L_A|)}.
\end{array}
\end{equation}
Since $A\in DE_n$ is irreducible, so is $\mu(D_{A}^{-1}A)\in DE_n$.
Then it follows from lemma \ref{le2'} that  $\mu(D_{A}^{-1}A)$ is
singular. As a result, (\ref{c1b1}) gives
$\det(e^{i\psi}I-H_{FGS})=0$ to reveal that $e^{i\psi}$ is an
eigenvalue of $H_{FGS}$. This completes the sufficiency.

The following prove the necessity. Let $e^{i\psi}$ is an eigenvalue
of $H_{FGS}$. Then
\begin{equation}\label{c1b2}
\begin{array}{lllll}
\det(e^{i\psi}I-H_{FGS})&=&\det(e^{i\psi}I-(D_A-L_A)^{-1}U_A)\\
&=&\displaystyle\frac{\det[e^{i\psi}(D_A-L_A)-U_A]}{\det(D_A-L_A)}\\
&=&0.\end{array}
\end{equation}
Thus, $\det(e^{i\psi}(D_A-L_A)-U_A)=0$ which shows that
$e^{i\psi}(D_A-L_A)-U_A$ is singular. Since
$e^{i\psi}(D_A-L_A)-U_A\in DE_n$ is irreducible for
$A=D_{A}-L_{A}-U_{A}\in DE_n$ is irreducible, it follows from lemma
\ref{le2'} shows that $I-D_{A}^{-1}L_A-e^{-i\psi}D_{A}^{-1}U_A\in
\mathscr{R}^{0}_n$. Theorem \ref{thgi0} shows that there exists an
unitary diagonal matrix $D$ such that
\begin{equation}\label{c1c1}
\begin{array}{llll}
&&D^{-1}(I-D_{A}^{-1}L_A-e^{-i\psi}D_{A}^{-1}U_A)D\\&=&I-D^{-1}(D_{A}^{-1}L_A)D-e^{-i\psi}D^{-1}(D_{A}^{-1}U_A)D\\
&=&I-|D_{A}^{-1}L_A|-|D_{A}^{-1}U_A|.
\end{array}
\end{equation}
Equality (\ref{c1c1}) shows $D^{-1}(D_{A}^{-1}L_A)D=|D_{A}^{-1}L_A|$
and $D^{-1}(D_{A}^{-1}U_A)D=e^{i\psi}|D_{A}^{-1}U_A|$. Therefore,
\begin{eqnarray*}D^{-1}(D_{A}^{-1}A)D&=&I-D^{-1}(D_{A}^{-1}L_A)D-D^{-1}(D_{A}^{-1}U_A)D\\&=&I-|D_{A}^{-1}L_A|-e^{i\psi}|D_{A}^{-1}U_A|,\end{eqnarray*}
that is, there exists an unitary diagonal matrix $D$ such that
$$D^{-1}(D_{A}^{-1}A)D^{-1}=I-|D_{A}^{-1}L_A|-e^{i\psi}|D_{A}^{-1}U_A|).$$
Theorem \ref{thgi0a1} shows that $D_{A}^{-1}A\in
\mathscr{U}^{\psi}_n$. Here, we finish the necessity.
\end{proof}

\begin{lemma}\label{le4'} Let $A=(a_{ij})\in DE_n\ (n\geq3)$ be irreducible.
Then $e^{i\phi}$ is an eigenvalue of $H_{BGS}$ if and only if
$D_{A}^{-1}A\in \mathscr{L}^{\phi}_n$, where
$D_{A}=diag(a_{11},a_{22},\cdots,a_{nn})$ and $\phi\in \mathbb{R}.$
\end{lemma}

\begin{proof} According to Theorem \ref{thgi0a2}, Theorem \ref{thgi0} and Lemma
\ref{le2'}, the proof is obtained immediately similar to the proof
of Lemma \ref{le4}. \end{proof}

\begin{theorem}\label{the5} Let $A\in DE_n\ (n\geq3)$ be
irreducible. Then $\rho(H_{FGS})<1$, where $H_{FGS}$ is defined in
(\ref{1j}), i.e., the sequence $\{x^{(i)}\}$ generated by FGS-scheme
(\ref{r5}) converges to the unique solution of {\rm(\ref{r1})} for
any choice of the initial guess $x^{(0)}$ if and only if
$D_{A}^{-1}A\notin {\mathscr{U}}^{\psi}_n$.
\end{theorem}

\begin{proof} The sufficiency can be proved by contradiction. We
assume that there exists an eigenvalue $\lambda$ of $H_{FGS}$ such
that $|\lambda|\geq 1$. Then
\begin{equation}\label{equ4'}
det(\lambda I-H_{FGS})=0.
\end{equation}
If $|\lambda|>1$, then $\lambda I-H_{FGS}=(D_A-L_A)^{-1}(\lambda
D_A-\lambda L_A-U_A)$. Obviously, $\lambda I-\lambda L-U\in ID_n$
and is nonsingular(see Theorem 1.21 in \cite{R.S.15}). As a result,
$det(\lambda I-H_{FGS})\neq 0$, which contradicts (\ref{equ4'}).
Thus, $|\lambda|=1$. Set $\lambda=e^{i\psi}$, where $\psi\in R$.
Then Lemma \ref{le4} shows that $D_{A}^{-1}A\in
\mathscr{U}^{\psi}_n$, which contradicts the assumption $A\notin
\mathscr{U}^{\theta}_n$. Therefore, $\rho(H_{FGS})<1$. The
sufficiency is finished.

Let us prove the necessity by contradiction. Assume that
$D_{A}^{-1}A\in {\mathscr{U}}^{\psi}_n$. It then follows from Lemma
\ref{le4} that $\rho(H_{FGS})=1$ which contradicts
$\rho(H_{FGS})<1$. A contradiction arise to demonstrate that the
necessity is true. Thus, we complete the proof.
\end{proof}

\begin{theorem}\label{the5'} Let $A\in DE_n\ (n\geq3)$ be
irreducible. Then $\rho(H_{BGS})<1$, where $H_{BGS}$ is defined in
(\ref{1j'}), i.e., the sequence $\{x^{(i)}\}$ generated by
BGS-scheme (\ref{r5}) converges to the unique solution of
{\rm(\ref{r1})} for any choice of the initial guess $x^{(0)}$ if and
only if $D_{A}^{-1}A\notin {\mathscr{L}}^{\phi}_n$.
\end{theorem}

\begin{proof} Similar to the proof of Theorem \ref{the5}, the proof is
obtained immediately by Lemma \ref{le4'}. \end{proof}

Following, the conclusions of Theorem \ref{the5} and Theorem
\ref{the5'} will be extended to irreducible matrices that belong to
the class of generalized diagonally equipotent matrices and the
class of irreducible mixed $H-$matrices.

\begin{theorem}\label{the6} Let $A=(a_{ij})\in GDE_n\ (n\geq3)$ be
irreducible. Then $\rho(H_{FGS})<1$, where $H_{FGS}$ is defined in
(\ref{1j}), i.e., the sequence $\{x^{(i)}\}$ generated by FGS-scheme
(\ref{r5}) converges to the unique solution of {\rm(\ref{r1})} for
any choice of the initial guess $x^{(0)}$ if and only if
$D_{A}^{-1}A\notin {\mathscr{U}}^{\psi}_n$.
\end{theorem}

\begin{proof}
According to Definition \ref{d2}, the exists a diagonal matrix
$E=diag(e_1,e_2,$ $\cdots,e_n)$ with $e_k>0$ for all $k\in \langle
n\rangle$, such that $AE=(a_{ij}e_j)\in DE_n.$ Let $AE=F=(f_{ij})$
with $f_{ij}=a_{ij}e_j$ for all $i,j\in \langle n\rangle$. Then
$H^F_{FGS}=E^{-1}H_{FGS}E$ and $D_F^{-1}F=E^{-1}(D_{A}^{-1}A)E$ with
$D_F=D_AE$. Theorem \ref{the5} yields that $\rho(H^F_{FGS})<1$ if
and only if $D_{F}^{-1}F\notin {\mathscr{U}}^{\psi}_n$. Since
$\rho(H^F_{FGS})=\rho(H_{FGS}$ and $D_{A}^{-1}A\notin
{\mathscr{U}}^{\psi}_n$ for $D_F^{-1}F=E^{-1}(D_{A}^{-1}A)E\notin
{\mathscr{U}}^{\psi}_n$ and $E=diag(e_1,e_2,$ $\cdots,e_n)$ with
$e_k>0$ for all $k\in \langle n\rangle$, $\rho(H_{FGS})<1$, i.e.,
the sequence $\{x^{(i)}\}$ generated by FGS-scheme (\ref{r5})
converges to the unique solution of {\rm(\ref{r1})} for any choice
of the initial guess $x^{(0)}$ if and only if $D_{A}^{-1}A\notin
{\mathscr{U}}^{\psi}_n$., i.e., $\rho(J_A)<1$ if and only if
$D_{A}^{-1}A\notin {\mathscr{U}}^{\psi}_n$.
\end{proof}

\begin{theorem}\label{the6'} Let $A=(a_{ij})\in GDE_n\ (n\geq3)$ be
irreducible. Then $\rho(H_{BGS})<1$, where $H_{BGS}$ is defined in
(\ref{1j}), i.e., the sequence $\{x^{(i)}\}$ generated by BGS-scheme
(\ref{r5}) converges to the unique solution of {\rm(\ref{r1})} for
any choice of the initial guess $x^{(0)}$ if and only if
$D_{A}^{-1}A\notin {\mathscr{L}}^{\phi}_n$.
\end{theorem}

\begin{proof}
Similar to the proof of Theorem \ref{the6}, we can obtain the proof
by Definition \ref{d2} and Theorem \ref{the5'}.
\end{proof}

According to Lemma \ref{0.1.7} and Lemma \ref{le0}, if a matrix is
an irreducible mixed $H-$matrix, then it is an irreducible
generalized diagonally equipotent matrix. As a consequence, we have
the following conclusions.

\begin{theorem}\label{the6''} Let $A=(a_{ij})\in H^M_n\ (n\geq3)$ be
irreducible. Then $\rho(H_{FGS})<1$, where $H_{FGS}$ is defined in
(\ref{1j}), i.e., the sequence $\{x^{(i)}\}$ generated by FGS-scheme
(\ref{r5}) converges to the unique solution of {\rm(\ref{r1})} for
any choice of the initial guess $x^{(0)}$ if and only if
$D_{A}^{-1}A\notin {\mathscr{U}}^{\psi}_n$.
\end{theorem}

\begin{theorem}\label{the6'''} Let $A=(a_{ij})\in H^M_n\ (n\geq3)$ be
irreducible. Then $\rho(H_{BGS})<1$, where $H_{BGS}$ is defined in
(\ref{1j}), i.e., the sequence $\{x^{(i)}\}$ generated by BGS-scheme
(\ref{r5}) converges to the unique solution of {\rm(\ref{r1})} for
any choice of the initial guess $x^{(0)}$ if and only if
$D_{A}^{-1}A\notin {\mathscr{L}}^{\phi}_n$.
\end{theorem}

Now, we consider convergence of SGS-iterative method. The following
lemma will be used in this section.

\begin{lemma}\label{le3.2} {\rm (see Lemma 3.13 in \cite{Z.X22})}~~Let $A=\left[
 \begin{array}{ccc}
 \ E & \ U\\
 \ L & \ F
 \end{array}
 \right]\in \mathbb{C}^{2n\times 2n}$, where $E,F,L,U\in C^{n\times n}$ and $E$ is nonsingular.
 Then the Schur
 complement of $A$ with respect to $E$,
 i.e., $A/E=F-LE^{-1}U$ is nonsingular if and only if $A$ is nonsingular.
\end{lemma}

\begin{theorem}\label{th3.24} Let $A\in DE_n\ (n\geq3)$ be
irreducible. Then $\rho(H_{SGS})<1$, where $H_{SGS}$ is defined in
(\ref{1jj}), i.e., the sequence $\{x^{(i)}\}$ generated by
SGS-scheme (\ref{r5}) converges to the unique solution of
{\rm(\ref{r1})} for any choice of the initial guess $x^{(0)}$ if and
only if $D_{A}^{-1}A\notin {\mathscr{R}}^{0}_n$.
\end{theorem}

\begin{proof}
The sufficiency can be proved by contradiction. We assume that there
exists an eigenvalue $\lambda$ of $H_{SGS}$ such that $|\lambda|\geq
1$. According to equality (\ref{1jj}),
\begin{equation}\label{da1}
\begin{array}{lllll}
\det(\lambda I-H_{SGS})&=&\det(\lambda I-(D_{A}-U_{A})^{-1}L_{A}(D_{A}-L_{A})^{-1}U_{A})\\
&=& \det[(D_{A}-U_{A})^{-1}]\det[\lambda (D_{A}-U_{A})-L_{A}(D_{A}-L_{A})^{-1}U_{A}]\\
&=&\displaystyle\frac{\det[\lambda (D_{A}-U_{A})-L_{A}(D_{A}-L_{A})^{-1}U_{A}]}{\det(D_{A}-U_{A})}\\
&=&0.
\end{array}
\end{equation}
Equality (\ref{da1}) gives
\begin{equation}\label{da2}
\det B=\det[\lambda (D_{A}-U_{A})-L_{A}(D_{A}-L_{A})^{-1}U_{A}]]=0,
\end{equation}
i.e., $B:=\lambda (D_{A}-U_{A})-L_{A}(D_{A}-L_{A})^{-1}U_{A}$ is
singular. Let $E=D_{A}-L_{A}$, $F=\lambda (D_{A}-U_{A})$ and
\begin{equation}\label{da3}
C=\left[\begin{array}{cc}
\ E &\ -U_{A}\\
-L_{A} &\ F
\end{array}
\right]=\left[\begin{array}{cc}
\ D_{A}-L_{A} &\ -U_{A}\\
-L_{A} &\ \lambda (D_{A}-U_{A})
\end{array}
\right].
\end{equation}
Then $B=F-L_{A}E^{-1}U_{A}$ is the Schur complement of $C$ with
respect to the principal submatrix $E$. Now, we investigate the
matrix $C$. Since $A$ is irreducible, both $L_{A}\neq 0$ and
$U_{A}\neq 0$. As a result, $C$ is also irreducible. If
$|\lambda|>1$, then (\ref{da3}) indicates $C\in ID_{2n}$.
Consequently, $C$ is nonsingular, so is $B=\lambda
(D_{A}-U_{A})-L_{A}(D_{A}-L_{A})^{-1}U_{A}$ coming from Lemma
\ref{le3.2}, i.e., $\det B\neq0$, which contradicts (\ref{da2}).
Therefore, $|\lambda|=1$. Let $\lambda=e^{i\theta}$ with $\theta\in
R$. (\ref{da2}) and Lemma \ref{le3.2} yield that
$C=\left[\begin{array}{cc}
\ D_{A}-L_{A} &\ -U_{A}\\
-L_{A} &\ e^{i\theta}(D_{A}-U_{A})
\end{array}
\right]$ and hence $C_1=\left[\begin{array}{cc}
\ D_{A}-L_{A} &\ -U_{A}\\
-e^{-i\theta}L_{A} &\ D_{A}-U_{A}
\end{array}
\right]$ are singular. Since $A=I-L-U\in DE_n$ and is irreducible,
both $C$ and $C_1$ are irreducible diagonally equipotent. The
singularity of $C_1$ and Lemma \ref{le2'} yield that
$D^{-1}_{C_1}C_1\in \mathscr{R}^{0}_{2n}$, where
$D_{C_1}=diag(D_{A},D_{A})$, i.e., there exists an $n\times n$
unitary diagonal matrix $D$ such that $\widetilde{D}=diag(D,D)$ and
\begin{equation}\label{da5}
\begin{array}{llll}
\widetilde{D}^{-1}(D^{-1}_{C_1}C_1)\widetilde{D}&=&\left[\begin{array}{cc}
\ I-D^{-1}(D^{-1}_{A}L_A)D &\ -D^{-1}(D^{-1}_{A}U_A)D\\
-e^{-i\theta}D^{-1}(D^{-1}_{A}L_A)D^{-1}_{A}L_AD &\
I-D^{-1}(D^{-1}_{A}U_A)D
\end{array}
\right]\\  \\&=&\left[\begin{array}{cc}
\ I-|D^{-1}_{A}L_A| &\ -|D^{-1}_{A}U_A|\\
-|D^{-1}_{A}L_A| &\ I-|D^{-1}_{A}U_A|
\end{array}
\right].
\end{array}
\end{equation}
(\ref{da5}) indicates that $\theta=2k\pi,$ where $k$ is an integer
and thus $\lambda=e^{i2k\pi}=1$, and there exists an $n\times n$
unitary diagonal matrix $D$ such that
$D^{-1}(D^{-1}_{A}A)D=I-|D^{-1}_{A}L_{A}|-|D^{-1}_{A}U_{A}|$, i.e.,
$D^{-1}_{A}A\in \mathscr{R}^{0}_n$. However, this contradicts
$D^{-1}_{A}A\notin \mathscr{R}^{0}_n$. Thus, $|\lambda|\neq1$.
According to the proof above, we have that $|\lambda|\geq1$ is not
true. Therefore, $\rho(H_{SGS})<1$, i.e., the sequence $\{x^{(i)}\}$
generated by SGS-scheme (\ref{r5}) converges to the unique solution
of {\rm(\ref{r1})} for any choice of the initial guess $x^{(0)}$.

The following will prove the necessity by contradiction. Assume that
$D^{-1}_{A}A\in \mathscr{R}^{0}_n$. Then there exists an $n\times n$
unitary diagonal matrix $D$ such that
$D^{-1}_{A}A=I-D^{-1}_{A}L_{A}-D^{-1}_{A}U_{A}=I-D|D^{-1}_{A}L_{A}|D^{-1}-D|D^{-1}_{A}U_{A}|D^{-1}$
and
\begin{equation}
\begin{array}{llll}
\label{da61}H_{SGS}&=&(D_{A}-U_{A})^{-1}L_{A}(D_{A}-L_A)^{-1}U_{A}\\
&=&[I-(D^{-1}_{A}U_{A})]^{-1}(D^{-1}_{A}L_{A})[I-(D^{-1}_{A}L_A)]^{-1}(D^{-1}_{A}U_{A})\\
&=&D[(I-|D^{-1}_{A}U|)^{-1}|D^{-1}_{A}L|(I-|D^{-1}_{A}L|)^{-1}|D^{-1}_{A}U|]D^{-1}.
\end{array}
\end{equation}
Hence,\\
$~~~~~~\det(I-{H}_{SGS})$
\begin{equation}\label{da7}
\begin{array}{lllll}
&=&\det\{I-D[(I-|D^{-1}_{A}U|)^{-1}|D^{-1}_{A}L|(I-|D^{-1}_{A}L|)^{-1}|D^{-1}_{A}U|]D^{-1}\}\\
&=&\det[I-(I-|D^{-1}_{A}U_{A}|)^{-1}|D^{-1}_{A}L_{A}|(I-|D^{-1}_{A}L_{A}|)^{-1}|D^{-1}_{A}U_{A}|]\\&=&
\displaystyle\frac{\det[(I-|D^{-1}_{A}U_{A}|)-|D^{-1}_{A}L_{A}|(I-|D^{-1}_{A}L_{A}|)^{-1}|D^{-1}_{A}U_{A}|]}{\det(I-|D^{-1}_{A}U_{A}|)}.
\end{array}
\end{equation}
Let $V=\left[\begin{array}{cc}
\ I-|D^{-1}_{A}L_{A}| &\ -|D^{-1}_{A}U_{A}|\\
-|D^{-1}_{A}L_{A}| &\ I-|D^{-1}_{A}U_{A}|
\end{array}
\right]$ and
$W=(I-|D^{-1}_{A}U_{A}|)-|D^{-1}_{A}L_{A}|(I-|D^{-1}_{A}L_{A}|)^{-1}|D^{-1}_{A}U_{A}|$.
Then $W$ is the Schur complement of $V$ with respect to
$I-|D^{-1}_{A}L_{A}|$.  Sice $A=I-L-U\in DE_n$ is irreducible,
$D^{-1}_{A}A=I-D^{-1}_{A}L_{A}-D^{-1}_{A}U_{A}\in DE_n$ is
irreducible. Therefore, $V\in DE_{2n}\cap\mathscr{R}^{0}_{2n}$ and
is irreducible. Lemma \ref{le2'} shows that $V$ is singular and
hence $$\det
V=\det[(I-|D^{-1}_{A}U_{A}|)-|D^{-1}_{A}L_{A}|(I-|D^{-1}_{A}L_{A}|)^{-1}|D^{-1}_{A}U_{A}|]=0.$$
Therefore, (\ref{da7}) yields $\det(I-{H}_{SGS})=0$, which shows
that $1$ is an eigenvalue of ${H}_{SGS}$. Thus,
$\rho(H_{SGS})\geq1$, i.e., the sequence $\{x^{(i)}\}$ generated by
SGS-scheme (\ref{r5}) doesn't converge to the unique solution of
{\rm(\ref{r1})} for any choice of the initial guess $x^{(0)}$. This
is a contradiction which shows that the assumption is incorrect.
Therefore, $A\notin\mathscr{R}^{0}_n$. This completes the proof.
\end{proof}

Lemma \ref{le2'} shows that the following corollary holds.

\begin{corollary}\label{coro3.24} Let $A\in DE_n\ (n\geq3)$ be
irreducible and nonsingular. Then $\rho(H_{SGS})<1$, where $H_{SGS}$
is defined in (\ref{1jj}), i.e., the sequence $\{x^{(i)}\}$
generated by SGS-scheme (\ref{r5}) converges to the unique solution
of {\rm(\ref{r1})} for any choice of the initial guess $x^{(0)}$.
\end{corollary}

\begin{theorem}\label{th3.24''} Let $A\in H^M_n\ (GDE_n)$ be
irreducible for $n\geq3$. Then $\rho(H_{SGS})<1$, where $H_{SGS}$ is
defined in (\ref{1jj}), i.e., the sequence $\{x^{(i)}\}$ generated
by SGS-scheme (\ref{r5}) converges to the unique solution of
{\rm(\ref{r1})} for any choice of the initial guess $x^{(0)}$ if and
only if $D_{A}^{-1}A\notin {\mathscr{R}}^{0}_n$.
\end{theorem}

\begin{proof}
According to Lemma \ref{0.1.7} and Lemma \ref{le0}, under the
condition of irreducibility, $H^M_n=GDE_n$. Then, similar to the
proof of Theorem \ref{the6}, we can obtain the proof by Definition
\ref{d2} and Theorem \ref{th3.24}.
\end{proof}

\begin{corollary}\label{th3.241} Let $A\in H^M_n\ (n\geq3)$ be
irreducible and nonsingular. Then $\rho(H_{SGS})<1$, where $H_{SGS}$
is defined in (\ref{1jj}), i.e., the sequence $\{x^{(i)}\}$
generated by SGS-scheme (\ref{r5}) converges to the unique solution
of {\rm(\ref{r1})} for any choice of the initial guess $x^{(0)}$.
\end{corollary}

\begin{proof}
It follows from Lemma \ref{le2'} that the proof of this corollary is
obtained immediately.
\end{proof}

In what follows we establish some convergence results on
Gauss-Seidel iterative methods for nonstrictly diagonally dominant
matrices.

\begin{theorem}\label{the6a1} Let $A=(a_{ij})\in D_n(GD_n)$ with $a_{ii}\neq 0$ for all $i\in \langle n\rangle$. Then $\rho(H_{FGS})<1$, where $H_{FGS}$
is defined in (\ref{1j}), i.e., the sequence $\{x^{(i)}\}$ generated
by FGS-scheme (\ref{r5}) converges to the unique solution of
{\rm(\ref{r1})} for any choice of the initial guess $x^{(0)}$ if and
only if $A$ has neither $2\times 2$ irreducibly (generalized)
diagonally equipotent principal submatrix nor irreducibly principal
submatrix $A_k=A(i_1,i_2,\cdots,i_k)$, $3\leq k\leq n$, such that
$D_{A_k}^{-1}A_k\notin {\mathscr{U}}^{\psi}_k\cap
DE_k$$({\mathscr{U}}^{\psi}_k\cap GDE_k)$, where
$D_{A_k}=diga(a_{i_1i_1},a_{i_2i_2},\cdots,a_{i_ki_k})$.
\end{theorem}

\begin{proof} The proof is obtained immediately by Theorem \ref{lemma8'}, Theorem \ref{le3}, Theorem \ref{the5} and Theorem \ref{the6}.
\end{proof}

\begin{theorem}\label{the6a2} Let $A=(a_{ij})\in D_n(GD_n)$ with $a_{ii}\neq 0$ for all $i\in \langle n\rangle$. Then $\rho(H_{BGS})<1$, where $H_{BGS}$
is defined in (\ref{1j}), i.e., the sequence $\{x^{(i)}\}$ generated
by BGS-scheme (\ref{r5}) converges to the unique solution of
{\rm(\ref{r1})} for any choice of the initial guess $x^{(0)}$ if and
only if $A$ has neither $2\times 2$ irreducibly (generalized)
diagonally equipotent principal submatrix nor irreducibly principal
submatrix $A_k=A(i_1,i_2,\cdots,i_k)$, $3\leq k\leq n$, such that
$D_{A_k}^{-1}A_k\notin {\mathscr{L}}^{\phi}_k\cap
DE_k$$({\mathscr{L}}^{\phi}_k\cap GDE_k)$, where
$D_{A_k}=diga(a_{i_1i_1},a_{i_2i_2},\cdots,a_{i_ki_k})$.
\end{theorem}

\begin{proof}
According to Theorem \ref{lemma8'}, Theorem \ref{le3}, Theorem
\ref{the5'} and Theorem \ref{the6'}, we can obtain the proof of this
theorem immediately.
\end{proof}

\begin{theorem}\label{the6a3} Let $A=(a_{ij})\in D_n(GD_n)$ with $a_{ii}\neq 0$ for all $i\in \langle n\rangle$. Then $\rho(H_{SGS})<1$, where $H_{SGS}$
is defined in (\ref{1j}), i.e., the sequence $\{x^{(i)}\}$ generated
by SGS-scheme (\ref{r5}) converges to the unique solution of
{\rm(\ref{r1})} for any choice of the initial guess $x^{(0)}$ if and
only if $A$ has neither $2\times 2$ irreducibly (generalized)
diagonally equipotent principal submatrix nor irreducibly principal
submatrix $A_k=A(i_1,i_2,\cdots,i_k)$, $3\leq k\leq n$, such that
$D_{A_k}^{-1}A_k\notin {\mathscr{R}}^{0}_k\cap
DE_k$$({\mathscr{R}}^{0}_k\cap GDE_k)$.
\end{theorem}

\begin{proof} It follows from Theorem \ref{lemma8'}, Theorem \ref{le3},
Theorem \ref{th3.24} and Theorem \ref{th3.24''} that the proof of this theorem is obtained immediately.
\end{proof}

\begin{theorem}\label{the7''} Let $A\in GD_n$ be nonsingular. Then $\rho(H_{SGS})<1$, where $H_{SGS}$ is defined in
(\ref{1jj}), i.e., the sequence $\{x^{(i)}\}$ generated by
SGS-scheme (\ref{r5}) converges to the unique solution of
{\rm(\ref{r1})} for any choice of the initial guess $x^{(0)}$ if and
only if $A$ has no $2\times 2$ irreducibly generalized diagonally
equipotent principal submatrices.
\end{theorem}

\begin{proof} Since $A\in GD_n$ is nonsingular, it follows from Theorem 3.11 in \cite{cyzhang01} that
$A$ hasn't any irreducibly principal submatrix
$A_k=A(i_1,i_2,\cdots,i_k)$, $3\leq k\leq n$, such that
$D_{A_k}^{-1}A_k\in {\mathscr{R}}^{0}_k$, and hence
$D_{A_k}^{-1}A_k\notin {\mathscr{R}}^{0}_k\cap GDE_k$. Then the
conclusion of this theorem follows Theorem \ref{the6a3}.
\end{proof}

In the rest of this section, the convergence results on Gauss-Seidel
iterative method for nonstrictly diagonally dominant matrices will
be extended to general $H-$matrices.

\begin{theorem}\label{the6a1a} Let $A=(a_{ij})\in H_n$ with $a_{ii}\neq 0$ for all $i\in \langle n\rangle$. Then $\rho(H_{FGS})<1$, where $H_{FGS}$
is defined in (\ref{1j}), i.e., the sequence $\{x^{(i)}\}$ generated
by FGS-scheme (\ref{r5}) converges to the unique solution of
{\rm(\ref{r1})} for any choice of the initial guess $x^{(0)}$ if and
only if $A$ has neither $2\times 2$ irreducibly generalized
diagonally equipotent principal submatrix nor irreducibly principal
submatrix $A_k=A(i_1,i_2,\cdots,i_k)$, $3\leq k\leq n$, such that
$D_{A_k}^{-1}A_k\notin {\mathscr{U}}^{\psi}_k\cap GDE_k$.
\end{theorem}

\begin{proof} If $A\in H_n$ is irreducible, it follows from Theorem \ref{le3} and Theorem \ref{the6''}
that the conclusion of this theorem is true. If $A\in H_n$ is
reducible, since $A\in H_n$ with $a_{ii}\neq 0$ for all $i\in
\langle n\rangle$, Theorem \ref{0.1.8} shows that each diagonal
square block $R_{ii}$ in the Frobenius normal from (\ref{eq5a}) of
$A$ is irreducible and generalized diagonally dominant for
$i=1,2,\cdots,s$. Let $H^{R_{ii}}_{FGS}$ denote the Gauss-Seidel
iteration matrix associated with diagonal square block $R_{ii}$.
Direct computations give $$\rho(H_{FGS})=\max\limits_{1\leq i\leq
s}\rho(H^{R_{ii}}_{FGS}).$$ Since $R_{ii}$ is irreducible and
generalized diagonally dominant, Theorem \ref{lemma8'}, Theorem
\ref{le3}, Theorem \ref{the5}, Theorem \ref{the6} and Theorem
\ref{the6a1} show that $\rho(H_{FGS})=\max\limits_{1\leq i\leq
s}\rho(H^{R_{ii}}_{FGS})<1$, i.e., the sequence $\{x^{(i)}\}$
generated by FGS-scheme (\ref{r5}) converges to the unique solution
of {\rm(\ref{r1})} for any choice of the initial guess $x^{(0)}$ if
and only if $A$ has neither $2\times 2$ irreducibly generalized
diagonally equipotent principal submatrix nor irreducibly principal
submatrix $A_k=A(i_1,i_2,\cdots,i_k)$, $3\leq k\leq n$, such that
$D_{A_k}^{-1}A_k\notin {\mathscr{U}}^{\psi}_k\cap GDE_k$.
\end{proof}

\begin{theorem}\label{the6a2a} Let $A=(a_{ij})\in H_n$ with $a_{ii}\neq 0$ for all $i\in \langle n\rangle$. Then $\rho(H_{BGS})<1$, where $H_{BGS}$
is defined in (\ref{1j}), i.e., the sequence $\{x^{(i)}\}$ generated
by BGS-scheme (\ref{r5}) converges to the unique solution of
{\rm(\ref{r1})} for any choice of the initial guess $x^{(0)}$ if and
only if $A$ has neither $2\times 2$ irreducibly generalized
diagonally equipotent principal submatrix nor irreducibly principal
submatrix $A_k=A(i_1,i_2,\cdots,i_k)$, $3\leq k\leq n$, such that
$D_{A_k}^{-1}A_k\notin {\mathscr{L}}^{\phi}_k\cap GDE_k$.
\end{theorem}

\begin{proof}
Similar to the proof of Theorem \ref{the6a1a}, we can obtain the
proof immediately by Theorem \ref{0.1.8} and Theorem \ref{the6a2}.
\end{proof}

\begin{theorem}\label{the6a3a} Let $A=(a_{ij})\in H_n$ with $a_{ii}\neq 0$ for all $i\in \langle n\rangle$. Then $\rho(H_{SGS})<1$, where $H_{SGS}$
is defined in (\ref{1j}), i.e., the sequence $\{x^{(i)}\}$ generated
by SGS-scheme (\ref{r5}) converges to the unique solution of
{\rm(\ref{r1})} for any choice of the initial guess $x^{(0)}$ if and
only if $A$ has neither $2\times 2$ irreducibly generalized
diagonally equipotent principal submatrix nor irreducibly principal
submatrix $A_k=A(i_1,i_2,\cdots,i_k)$, $3\leq k\leq n$, such that
$D_{A_k}^{-1}A_k\notin {\mathscr{R}}^{0}_k\cap GDE_k$.
\end{theorem}

\begin{proof}
Similar to the proof of Theorem \ref{the6a1a}, we can obtain the
proof immediately by Theorem \ref{0.1.8} and Theorem \ref{the6a3}.
\end{proof}

\begin{theorem}\label{th2} Let $A\in H_n$ be nonsingular. Then $\rho(H_{SGS})<1$, where $H_{SGS}$ is defined in
(\ref{1jj}), i.e., the sequence $\{x^{(i)}\}$ generated by
SGS-scheme (\ref{r5}) converges to the unique solution of
{\rm(\ref{r1})} for any choice of the initial guess $x^{(0)}$ if and
only if $A$ has no $2\times 2$ irreducibly generalized diagonally
equipotent principal submatrices.
\end{theorem}

\begin{proof} The proof is similar to the proof of Theorem
\ref{the7''}.
\end{proof}

The research in this section shows that the FGS iterative method
associated with the irreducible matrix $A\in H_n^M\cap
\mathscr{U}^{\theta}_n$ fails to converge, the same does for the BGS
iterative method associated with the irreducible matrix $A\in
H_n^M\cap \mathscr{L}^{\theta}_n$ and the SGS iterative method
associated with the irreducible matrix $A\in H_n^M\cap
\mathscr{R}^{0}_n$. It is natural to consider convergence on
preconditioned Gauss-Seidel iterative methods for nonsingular
general $H-$matrices.

\section{Convergence on preconditioned Gauss-Seidel iterative
methods}\label{further-sec} In this section, Gauss-type
preconditioning techniques for linear systems with nonsingular
general $H-$matrices are chosen such that the coefficient matrices
are invertible $H-$matrices. Then based on structure heredity of the
Schur complements for general $H-$matrices in \cite{cyzhang01},
convergence on preconditioned Gauss-Seidel iterative methods will be
studied and some results will be established.

Many researchers have considered the left Gauss-type preconditioner
applied to linear system (\ref{r1}) such that the associated Jacobi
and Gauss-Seidel methods converge faster than the original ones.
Milaszewicz \cite{milaszewicz1987} considered the preconditioner
\begin{equation}\label{precondition1} {P}_1=\left[
\begin{array}{cccc}
 \ 1 & \ 0 & \cdots & 0 \\
 \ -a_{21}  & \ 1 & \cdots & 0 \\
 \ \vdots & \vdots & \ddots & \vdots \\
 \ -a_{n1}  & \ 0 & \cdots & 1
 \end{array}
 \right].
 \end{equation}
 Later, Hadjidimos et al \cite{hadjidimos2003} generalized Milaszewicz's preconditioning
 technique and presented the preconditioner
\begin{equation}\label{precondition2} {P}_1(\alpha)=\left[
\begin{array}{cccc}
 \ 1 & \ 0 & \cdots & 0 \\
 \ -\alpha_2a_{21}  & \ 1 & \cdots & 0 \\
 \ \vdots & \vdots & \ddots & \vdots \\
 \ -\alpha_na_{n1}  & \ 0 & \cdots & 1
 \end{array}
 \right].
 \end{equation}
Recently, Zhang et al. \cite{zhang2005} proposed the left Gauss type
preconditioning techniques which utilizes the Gauss transformation
\cite{G.H} matrices as the base of the Gauss typ precondtioner based
on Hadjidimos et al. \cite{hadjidimos2003}, Milaszewicz
\cite{milaszewicz1987} and $LU$ factorization method \cite{G.H}. The
construction of Gauss transformation matrices is as follows:
\begin{equation}\label{precondition3} M_k=\left[
\begin{array}{cccccc}
 \ 1 & \cdots & 0 & 0 &\cdots & 0 \\
 \ \vdots &  &  \vdots &  \vdots &  &\vdots\\
 \ 0 & \cdots & 1 & 0 &\cdots & 0 \\
 \ 0 & \cdots & -\tau_{k+1} & 1 & \cdots & 0\\
 \ \vdots & &\vdots & \vdots & & \vdots \\
 \ 0  & \cdots & -\tau_{n} & 0 &\cdots & 1
 \end{array}
 \right],
 \end{equation}
 where $\tau_{i}=a_{ik}/a_{kk},\ i=k+1,\cdots,n$ and
 $k=1,2,\cdots,n-1.$ Zhang  et al. \cite{zhang2005} consider the
 following left preconditioners:
\begin{equation}\label{precondition4}
\mathscr{P}_1=M_1,\ \mathscr{P}_2=M_2M_1,\ \cdots,\
\mathscr{P}_{n-1}=M_{n-1}M_{n-2}\cdots M_2M_1.
 \end{equation}

Let $\mathscr{H}_n=\{A\in H_n:\ A~is~nonsingular\}$. Then
$H^I_n\subset\mathscr{H}_n$ while $H^I_n\neq\mathscr{H}_n$. Again,
let $\hat{H}^M_n=\{A\in H^M_n:\ A~is~nonsingular\}$. In fact,
$\mathscr{H}_n=H^I_n\cup\hat{H}^M_n.$ Thus, nonsingular general
$H-$matrices that the matrices in $\mathscr{H}_n$ differ from
invertible $H-$matrices. In this section we will propose some
Gauss-type preconditioning techniques for linear systems with the
coefficient matrices belong to $\mathscr{H}_n$ and establish some
convergence results on preconditioned Gauss-Seidel iterative
methods.

Firstly, we consider the case that the coefficient matrix
$A\in\mathscr{H}_n$ is irreducible. Then let us generalize the
preconditioner of (\ref{precondition1}),(\ref{precondition2}) and
(\ref{precondition3}) as follows:
\begin{equation}\label{precondition5} \mathscr{P}_k=\left[
\begin{array}{ccccccc}
 \ 1 & \cdots & 0& -\tau_{1} & 0 &\cdots & 0 \\
 \ \vdots &  & \vdots & \vdots &  \vdots &  &\vdots\\
 \ 0 & \cdots & 1 & -\tau_{k-1}& 0 & \cdots & 0\\
 \ 0 & \cdots & 0 & 1 & 0 &\cdots & 0 \\
 \ 0 & \cdots & 0 & -\tau_{k+1} & 1 & \cdots & 0\\
 \ \vdots & &\vdots& \vdots & \vdots & & \vdots \\
 \ 0  & \cdots & 0 & -\tau_{n} & 0 &\cdots & 1
 \end{array}
 \right],
 \end{equation}
 where $\tau_{i}=a_{ik}/a_{kk},\ i=1,\cdots,n$; $i\neq k$ and
 $k\in \langle n\rangle.$ Assume that $\tilde{A}_k=\mathscr{P}_kA$
 for $k\in \langle n\rangle$, ${H}^A_J,~{H}^A_{FGS},~{H}^A_{BGS}$ and
 ${H}^A_{SGS}$ denote the Jacobi and the forward, backward and symmetric Gauss-Seidel (FGS-, BGS- and
SGS-) iteration matrices associated with the coefficient matrix
${A}$, respectively.

\begin{theorem}\label{th5z} Let $A\in \mathscr{H}_n$ be irreducible.
Then $\tilde{A}_k=\mathscr{P}_kA\in H^I_n$ for all $k\in \langle
n\rangle$, where $\mathscr{P}_k$ is defined in
(\ref{precondition5}). Furthermore, the following conclusions hold:
\begin{enumerate}
\item $\rho({H}^{\tilde{A}_k}_J)\leq\rho({H}^{\mu(A/k)}_J)<1$ for all $k\in
\langle n\rangle$, where $A/k=A/\alpha$ with $\alpha=\{k\}$;
\item $\rho({H}^{\tilde{A}_k}_{FGS})\leq\rho({H}^{\mu(A/k)}_{FGS})<1$ for all $k\in
\langle n\rangle$;
\item $\rho({H}^{\tilde{A}_k}_{BGS})\leq\rho({H}^{\mu(A/k)}_{BGS})<1$ for all $k\in
\langle n\rangle$;
\item $\rho({H}^{\tilde{A}_k}_{SGS})\leq\rho({H}^{\mu(A/k)}_{SGS})<1$ for all $k\in
\langle n\rangle$,
\end{enumerate}
i.e., the sequence $\{x^{(i)}\}$ generated by the preconditioned
Jacobi, FGS, BGS and SGS iterative schemes (\ref{r5}) converge to
the unique solution of {\rm(\ref{r1})} for any choice of the initial
guess $x^{(0)}$.
\end{theorem}

\begin{proof} Since $A\in {H}^M_n$ is irreducible and nonsingular for $A\in \mathscr{H}_n$ is
irreducible, it follows from Theorem 5.9 in \cite{cyzhang01} that
$A/\alpha$ is an invertible $H-$matrix, where $\alpha=\{k\}$. For
the preconditioner $\mathscr{P}_k$, there exists a permutation
matrix $P_k$ such that
 $P_k\mathscr{P}_kP_k^T=\left[
\begin{array}{cc}
 \ 1 & 0 \\
 \ -\tau & I_{n-1}
 \end{array}
 \right]$, where $\tau=(\tau_1,\cdots,\tau_{k-1},\tau_{k+1},\cdots,\tau_n)^T$. As a consequence,
$$P_k(\mathscr{P}_kA)P_k^T=P_k\mathscr{P}_kP_k^TP_kAP_k^T=\left[
\begin{array}{cc}
 \ a_{kk} & \alpha_k \\
 \ 0 & A/\alpha
 \end{array}
 \right]$$ is an invertible $H-$matrix, so is $\mathscr{P}_kA$.
 Following, Theorem 4.1 in \cite{cyzhang01} and Theorem \ref{le2}
 show that the four conclusions hold. This completes the proof.
\end{proof}

On the other hand, if an irreducible matrix $A\in\mathscr{H}_n$ has
a principal submaitrix $A(\alpha)$ which is easy to get its inverse
matrix or is a (block)triangular matrix, there exists a permutation
matrix $P_{\alpha}$ such that

\begin{equation}\label{r1a}P_{\alpha}AP_{\alpha}^T=\left[
 \begin{array}{ccccc}
 \ A(\alpha) & \ A(\alpha,\alpha')  \\
 \ A(\alpha',\alpha) & \ A(\alpha',\alpha')
 \end{array}
 \right],
\end{equation}
where $\alpha'=\langle n\rangle-\alpha$. Let
\begin{equation}\label{r4a}
 M=\left[
 \begin{array}{ccccc}
 \ I_{|\alpha|} & 0)  \\
 \ -[A(\alpha)]^{-1}A(\alpha',\alpha) & \ I
 \end{array}
 \right].
 \end{equation} Then \begin{equation}\label{r2a}MP_{\alpha}AP_{\alpha}^T=\left[
 \begin{array}{ccccc}
 \ A(\alpha) & \ A(\alpha,\alpha')  \\
 \ 0 & \ A/\alpha
 \end{array}
 \right],
\end{equation}
where $A(\alpha)$ and $A/\alpha$ are both invertible $H-$matrices,
so is $MP_{\alpha}AP_{\alpha}^T$. As a result,
$P_{\alpha}^TMP_{\alpha}A=P^T(MP_{\alpha}AP_{\alpha}^T)P$ is an
invertible $H-$matrix. Therefore, we consider the following
preconditioner
\begin{equation}\label{r5a}\mathscr{P}_{\alpha}=P_{\alpha}^TMP_{\alpha},
\end{equation}
where $P_{\alpha}$ and $M$ are defined by (\ref{r1a}) and
(\ref{r1a}), respectively.

\begin{theorem}\label{th6z} Let $A\in \mathscr{H}_n$ be irreducible.
Then $\tilde{A}_{\alpha}=\mathscr{P}_{\alpha}A\in H^I_n$ for all
$\alpha\subset\langle n\rangle,~\alpha\neq\emptyset$, where
$\mathscr{P}_{\alpha}$ is defined in (\ref{r5a}). Furthermore, the
following conclusions hold:
\begin{enumerate}
\item $\rho({H}^{\tilde{A}_{\alpha}}_J)\leq\max\{\rho({H}^{\mu(A(\alpha))}_J),\rho({H}^{\mu(A/{\alpha})}_J)\}<1$ for all ${\alpha}\in
\langle n\rangle$;
\item $\rho({H}^{\tilde{A}_{\alpha}}_{FGS})\leq\max\{\rho({H}^{\mu(A(\alpha))}_{FGS}),\rho({H}^{\mu(A/{\alpha})}_{FGS})\}<1$ for all ${\alpha}\in
\langle n\rangle$;
\item $\rho({H}^{\tilde{A}_{\alpha}}_{BGS})\leq\max\{\rho({H}^{\mu(A(\alpha))}_{BGS}),\rho({H}^{\mu(A/{\alpha})}_{BGS})\}<1$ for all ${\alpha}\in
\langle n\rangle$;
\item $\rho({H}^{\tilde{A}_{\alpha}}_{SGS})\leq\max\{\rho({H}^{\mu(A(\alpha))}_{SGS}),\rho({H}^{\mu(A/{\alpha})}_{SGS})\}<1$ for all ${\alpha}\in
\langle n\rangle$,
\end{enumerate}
i.e., the sequence $\{x^{(i)}\}$ generated by the preconditioned
Jacobi, FGS, BGS and SGS iterative schemes (\ref{r5}) converge to
the unique solution of {\rm(\ref{r1})} for any choice of the initial
guess $x^{(0)}$.
\end{theorem}

\begin{proof}
The proof is similar to the proof of Theorem \ref{th5z}.
\end{proof}

Following, we consider the case that the coefficient matrix
$A\in\mathscr{H}_n$ is reducible. If there exists a proper
$\alpha=\langle n\rangle-\alpha'\subset\langle n\rangle$ such that
$A(\alpha)$ and $A(\alpha')$ are both invertible $H-$matrices, we
consider the preconditioner (\ref{r5a}) and have the following
conclusion.

\begin{theorem}\label{th7z} Let $A\in \mathscr{H}_n$ and a proper
$\alpha=\langle n\rangle-\alpha'\subset\langle
n\rangle,~\alpha\neq\emptyset,$ such that $A(\alpha)$ and
$A(\alpha')$ are both invertible $H-$matrices. Then
$\tilde{A}_{\alpha}=\mathscr{P}_{\alpha}A\in H^I_n$, where
$\mathscr{P}_{\alpha}$ is defined in (\ref{r5a}). Furthermore, the
following conclusions hold:
\begin{enumerate}
\item $\rho({H}^{\tilde{A}_{\alpha}}_J)\leq\max\{\rho({H}^{\mu(A(\alpha))}_J),\rho({H}^{\mu(A/{\alpha})}_J)\}<1$ for all ${\alpha}\in
\langle n\rangle$;
\item $\rho({H}^{\tilde{A}_{\alpha}}_{FGS})\leq\max\{\rho({H}^{\mu(A(\alpha))}_{FGS}),\rho({H}^{\mu(A/{\alpha})}_{FGS})\}<1$ for all ${\alpha}\in
\langle n\rangle$;
\item $\rho({H}^{\tilde{A}_{\alpha}}_{BGS})\leq\max\{\rho({H}^{\mu(A(\alpha))}_{BGS}),\rho({H}^{\mu(A/{\alpha})}_{BGS})\}<1$ for all ${\alpha}\in
\langle n\rangle$;
\item $\rho({H}^{\tilde{A}_{\alpha}}_{SGS})\leq\max\{\rho({H}^{\mu(A(\alpha))}_{SGS}),\rho({H}^{\mu(A/{\alpha})}_{SGS})\}<1$ for all ${\alpha}\in
\langle n\rangle$,
\end{enumerate}
i.e., the sequence $\{x^{(i)}\}$ generated by the preconditioned
Jacobi, FGS, BGS and SGS iterative schemes (\ref{r5}) converge to
the unique solution of {\rm(\ref{r1})} for any choice of the initial
guess $x^{(0)}$.
\end{theorem}

\begin{proof}
It is obvious that here exists a permutation matrix $P_{\alpha}$
such that (\ref{r1a}) holds. Further,
\begin{equation}\label{gauss1}
\mathscr{P}_{\alpha}A=P_{\alpha}^TMP_{\alpha}A=P^T(MP_{\alpha}AP_{\alpha}^T)P=P^T\left[
 \begin{array}{ccccc}
 \ A(\alpha) & \ A(\alpha,\alpha')  \\
 \ 0 & \ A/\alpha
 \end{array}
 \right]P.
\end{equation} Since $A\in \mathscr{H}_n$, $A\in H^I_n\cup H^M_n$ is
nonsingular. Again, $A(\alpha)$ and $A(\alpha')$ are both invertible
$H-$matrices, it follows from Theorem 5.2 and Theorem 5.11 in
\cite{cyzhang01} that $A/\alpha$ is an invertible $H-$matrix.
Therefore, $\tilde{A}_{\alpha}=\mathscr{P}_{\alpha}A\in H^I_n$
coming from (\ref{gauss1}). Following, Theorem 4.1 in
\cite{cyzhang01} and Theorem \ref{le2}
 yield that the four conclusions hold, which completes the proof.
\end{proof}

It is noted that the preconditioner $\mathscr{P}_{\alpha}$ has at
least two shortcomings when the coefficient matrix
$A\in\mathscr{H}_n$ is reducible. One is choice of $\alpha$. For a
large scale reducible matrix $A\in\mathscr{H}_n\cap H_n^M$, we are
not easy to choose $\alpha$ such that $A(\alpha)$ and $A(\alpha')$
are both invertible $H-$matrices. The other is the computation of
$[A(\alpha)]^{-1}$. Although $A(\alpha)$ is an invertible
$H-$matrices, it is difficult to obtain its inverse matrix for large
$A(\alpha)$. These shortcomings above are our further research
topics.

\section{Numerical examples} \label{examples-sec}
In this section some examples are given to illustrate the results
obtained in Section 4 and Section 5.

\begin{example} {\rm Let the coefficient matrix $A$ of linear system
(\ref{r1}) be given by the following $n\times n$ matrix}
\end{example}
\begin{equation}\label{bs1}
 \ A_n=\left[
 \begin{array}{cccccccc}
 \ 1  & -1 & 0 & 0 & \cdots & 0 & 0 & 0 \\
 \ 1  & 2 & -1 & 0 & \cdots & 0 & 0 & 0\\
 \ 0 & 1 & 2 & -1 & \cdots & 0 & 0 & 0\\
 \ 0 & 0 & 1 &  2 &\ddots & 0 & 0 & 0\\
\ \vdots & \vdots & \vdots &  \ddots & \ddots & \ddots & \vdots & \vdots\\
\ 0 & 0 & 0 &  0 &\ddots & 2 & -1 & 0\\
\ 0  & 0 & 0 & 0 & \cdots & 1 & 2 & -1 \\
 \ 0  & 0 & 0 & 0 & \cdots & 0 & 1 & 1
 \end{array}
 \right].
 \end{equation}

It is easy to see that $A_n\in DE_n\subset H_n$ is irreducible and $A_n\notin
 H^I_n$, but Lemma 4.3 in \cite{Z.X19} shows that $A_n$ is nonsingular. Thus, $A_n\in \mathscr{H}_n$ is irreducible. Since $$D_n^{-1}A_nD_n=|D_{A_n}|-|L_{A_n}|-e^{i\pi}|U_{A_n}|,$$ where $$D_n=diag[1,-1,\cdots,(-1)^{k-1},\cdots,(-1)^{n-1}],$$ it follows from Theorem 3.6 that $A_n\in {\mathscr{U}}^{\pi}_n$. In addition, it is obvious that $A_n\in {\mathscr{L}}^{\pi}_n$. Therefore, Theorem 4.9 and Theorem 4.10 show that $$\rho(H_{FGS}(A_{100}))=\rho(H_{BGS}(A_{100}))=1.$$ Futher, Theorem 4.16 shows that $\rho(H_{SGS}(A_{100}))<1.$  In fact, direct computations also get $\rho(H_{FGS}(A_{100}))=\rho(H_{BGS}(A_{100}))=1$ and $\rho(H_{SGS}(A_{100}))=0.3497<1$ which demonstrates that the conclusions of Theorem 4.9, Theorem 4.10 and Theorem 4.16 in Section 4 are correct and effective.

 The discussion above shows that FGS and BGS iterative schemes fail to converge to
the unique solution of linear system (\ref{r1}) with the coefficient matrix (\ref{bs1}) of  for any choice of the initial
guess $x^{(0)}$, but SGS iterative schemes does. Now we consider preconditioned Gauss-Seidel iterative methods for linear system (\ref{r1}) with the coefficient matrix (\ref{bs1}).

Choose two set $\alpha=\{1\}\in \langle n\rangle$ and $\beta=\{1,n\}\in \langle n\rangle$ and partition $A_n$ into
\begin{equation}\label{example1}
A_n=\left[
 \begin{array}{ccccc}
 \ 1 & \ -a^T  \\
 \ a & \ A_{n-1}
 \end{array}
 \right]=\left[
 \begin{array}{ccccc}
 \ 1 & \ -b^T & 0 \\
 \ b & \ A_{n-2} & -c^T\\
 \ 0 & \ c       &  1
 \end{array}
 \right],
\end{equation} where $a=(1,0,\cdots,0)^T\in \mathbb{R}^{n-1},$ $b=(1,0,\cdots,0)^T\in \mathbb{R}^{n-2}$, $c^T=(0,\cdots,0,1)^T\in \mathbb{R}^{n-2}$ and $A_{n-2}=tri[1,2,-1]\in \mathbb{R}^{(n-2)\times(n-2)},$
we get two preconditioners
\begin{equation}\label{example2}
\mathscr{P}_1=\left[
 \begin{array}{ccccc}
 \ 1 & \ 0  \\
 \ -a & \ I_{n-1}
 \end{array}
 \right]\ \ and \ \ \mathscr{P}_{\beta}=\left[
 \begin{array}{ccccc}
 \ 1 & \ 0 & 0 \\
 \ -b & \ I_{n-2} & c^T\\
 \ 0 & \ c       &  1
 \end{array}
 \right],
\end{equation} where $I_{n-1}$ is the $(n-1)\times(n-1)$ identity matrix.
Then Theorem 5.9 in \cite{cyzhang01} shows that $\tilde{A}_1=\mathscr{P}_1A_n=\left[
 \begin{array}{ccccc}
 \ 1 & \ -a^T  \\
 \ 0 & \ A_{n}/\alpha
 \end{array}
 \right]$ and $\tilde{A}_{\beta}=\mathscr{P}_{\beta}A_n=\left[
 \begin{array}{ccccc}
 \ 1 & \ 0 & 0 \\
 \ 0 & \ A_{n}/\beta & 0\\
 \ 0 & \ c       &  1
 \end{array}
 \right]$ are both invertible $H-$matrices. According to Theorem 5.1 and Theorem 5.2, for these two preconditioners, the preconditioned
FGS, BGS and SGS iterative schemes converge to
the unique solution of {\rm(\ref{r1})} for any choice of the initial
guess $x^{(0)}$.

In fact, by direct computations, Table 6.1 in the following is obtained to show that $\rho({H}^{\tilde{A}_1}_{FGS})=\rho[{H}^{\mu(\tilde{A}_1)}_{FGS}]=0.9970<1$,
$\rho({H}^{\tilde{A}_{\beta}}_{BGS})=\rho[{H}^{\mu(\tilde{A}_{\beta})}_{BGS}]=0.9970<1$ and $\rho({H}^{\tilde{A}_1}_{SGS})=0.3333<\rho[{H}^{\mu(\tilde{A}_1)}_{SGS}]=0.9950<1$, $\rho({H}^{\tilde{A}_{\beta}}_{SGS})=0.3158<\rho[{H}^{\mu(\tilde{A}_{\beta})}_{SGS}]=0.9979<1$,
which illustrate specifically that Theorem 5.1 and Theorem 5.2 are both valid.

\begin{center}{\hbox{Table 6.1\ \ The comparison result of spectral radii of PGS iterative matrices}}  { {
\begin{tabular}{c|c|c|c|c}
$X$& $\rho({H}^{\tilde{A}_1}_{X})$& $\rho[{H}^{\mu(\tilde{A}_1)}_{X}]$ & $\rho({H}^{\tilde{A}_{\beta}}_{X})$ &
$\rho[{H}^{\mu(\tilde{A}_{\beta})}_{X}]$\\
\hline $FGS$ & $0.9970$ & $0.9970$ & $0.9900$&
$0.9900$  \\
\hline{$BFS$} &$0.9970$&$0.9970$&$0.9900$&$0.9900$ \\
\hline $SGS$ & $0.3333$&$0.9950$&$0.3158$& $0.9979$
\end{tabular}} }
\end{center}

\begin{example} {\rm Let the coefficient matrix $A$ of linear system
(\ref{r1}) be given by the following $6\times 6$ matrix}
\end{example}
\begin{equation}\label{bs2}
 \ A=\left[
 \begin{array}{cccccccc}
 \ 5  & -1 & 1 & 1 & 1 & -1 \\
 \ 1  & 5 & -1 & 1 & 1 & 1 \\
 \ 1 & 1 & 5 & -1 & 1 & 1  \\
 \ 0 & 0 & 0 &  2 & -1 & 1 \\
\ 0 & 0 & 0 &  1 & 2 & -1 \\
\ 0 & 0 & 0 &  1 & 1 & 2
 \end{array}
 \right].
 \end{equation}

Although $A\in DE_6$ are reducible but there is not any principal submatrix $A_k$ ($k<6$) in $A$ such that $D_{A_k}^{-1}A_k\in {\mathscr{R}}^{0}_k$, Theorem 3.16 in \cite{cyzhang01} shows that $A$ is nonsingular. Thus, $A_n\in \mathscr{H}_6$ is reducible. Furthermore, there is not any principal submatrix $A_k$ in $A$ such that $D_{A_k}^{-1}A_k\in {\mathscr{U}}^{\psi}_k$ and $D_{A_k}^{-1}A\in {\mathscr{L}}^{\phi}_k$. It follows from Theorem 4.20, Theorem 4.21 and Theorem 4.22 that FGS, BGS and SGS iterative schemes converge to the unique solution of {\rm(\ref{r1})} for any choice of the initial guess $x^{(0)}$.

From the first column in Table 6.2, one has $\rho({H}_{FGS})=\rho({H}_{BGS})=0.3536<1$ and $\rho({H}_{SGS})=0.2500<1$. This naturally verifies the results of Theorem 4.20, Theorem 4.21 and Theorem 4.22.

\begin{center}{\hbox{Table 6.2\ \ The comparison result of spectral radii of GS and PGS iterative matrices }}  { {
\begin{tabular}{c|c|c|c}
$X$& $\rho({H}_{X})$ & $\rho({H}^{\tilde{A}_{\alpha}}_{X})$ &
$\rho[{H}^{\mu(\tilde{A}_{\alpha})}_{X}]$\\
\hline $FGS$ & $0.3536$ & $0.6000$ &
$0.6000$  \\
\hline{$BFS$} &$0.3536$&$0.6000$&$0.6000$ \\
\hline $SGS$ & $0.2500$&$0.6000$&$0.6000$
\end{tabular}} }
\end{center}

Now, we consider convergence on preconditioned Gauss-Seidel iterative methods. Set $\alpha=\{3,4\}\subset\langle6\rangle=\{1,2,3,4,5,6\}$, and set $\beta=\{3,4\}\subset\langle6\rangle$ and $\gamma=\{3,4\}\subset\langle6\rangle$. Since $A(\beta\cup\gamma)\in H^I_4$, it follows from Theorem 4.3 in \cite{Z.X22} that $A/\alpha\in H^I_4$. Thus, we choose a preconditioner
 \begin{equation}\label{example3}
 \mathscr{P}_{\alpha}=\left[
 \begin{array}{ccccc}
 \ I_2 & \ -A(\beta,\alpha)[A(\alpha)]^{-1} & 0 \\
 \ 0 & \ I_{2} & c^T\\
 \ 0 & \ -A(\gamma,\alpha)[A(\alpha)]^{-1}    &  I_2
 \end{array}
 \right]
\end{equation}
such that $\tilde{A}_{\alpha}=\mathscr{P}_{\alpha}A\in H^I_6$. From Theorem 5.3, it is obvious to see that the preconditioned
FGS, BGS and SGS iterative schemes (\ref{r5}) converge to
the unique solution of {\rm(\ref{r1})} for any choice of the initial
guess $x^{(0)}$.

As is shown in Table 6.2, $\rho({H}^{\tilde{A}_{\alpha}}_{FGS})=\rho[{H}^{\mu(\tilde{A}_{\alpha})}_{FGS}]=0.6000<1$, $\rho({H}^{\tilde{A}_{\alpha}}_{BGS})=\rho[{H}^{\mu(\tilde{A}_{\alpha})}_{BGS}]=0.6000<1$ and $\rho({H}^{\tilde{A}_{\alpha}}_{SGS})=\rho[{H}^{\mu(\tilde{A}_{\alpha})}_{SGS}]=0.6000<1$, which directly verifies the results of Theorem 5.3.


\section{Conclusions} \label{conclusions-sec}

This paper studies convergence on Gauss-Seidel iterative methods for
nonstrictly diagonally dominant matrices and general $H-$matrices.
The definitions of some special matrices are firstly proposed to
establish some new results on convergence of Gauss-Seidel iterative
methods for nonstrictly diagonally dominant matrices and general
$H-$matrices. Following, convergence of Gauss-Seidel iterative
methods for preconditioned linear systems with general $H-$matrices
is established. Finally, some numerical examples are given to
demonstrate the results obtained in this paper.



\bigskip
{\bf Acknowledgment.} The authors would like to thank the anonymous
referees for their valuable comments and suggestions, which actually
stimulated this work.


\end{document}